\documentclass[11pt]{amsart}

\usepackage[english,activeacute]{babel}
\usepackage[latin1]{inputenc}
\usepackage{graphicx}
\usepackage{url}
\usepackage{babel}
\usepackage{latexsym}
\usepackage{euscript}
\usepackage{amsfonts}
\usepackage{amssymb}
\usepackage{enumerate}
\usepackage{amsmath}
\usepackage{amsthm}
\usepackage{bbm}
\usepackage{color}
\usepackage{setspace}
\usepackage{hyperref}

\numberwithin{equation}{section}
\theoremstyle{plain}
\newtheorem{Teo}{Theorem}

\newtheorem{Prop}[Teo]{Proposition}
\newtheorem{Lemma}[Teo]{Lemma}
\newtheorem{Corollary}[Teo]{Corollary}
\newtheorem{Rem}{Remark}

\providecommand{\abs}[1]{\left\vert #1 \right\vert}

\providecommand{\norm}[1]{\|#1\|}
\providecommand{\parred}[1]{\left( #1 \right)}
\providecommand{\parcuad}[1]{\left[ #1 \right]}

\providecommand{\dr}[2]{\frac{\partial #1}{\partial #2}}
\providecommand{\ddr}[2]{\frac{\partial^2 #1}{\partial #2^2}}

\newcommand{\be}[1]{\begin{equation}\label{#1}}
\newcommand{\ee}{\end{equation}}
\newcommand{\nrm}[2]{\left\|#1\right\|_{L^{#2}(\R^2)}}
\newcommand{\beq}{\begin{eqnarray}}
\newcommand{\eeq}{\end{eqnarray}}
\newcommand{\ben}{\begin{eqnarray*}}
\newcommand{\een}{\end{eqnarray*}}
\renewcommand{\(}{\left(}
\renewcommand{\)}{\right)}
\newcommand{\ix}[1]{\int_{\R^2}{#1}\;dx}
\newcommand{\ir}[1]{\int_{0}^{\infty}{#1}\;ds}

\newcommand{\ninf}{n_\infty}
\newcommand{\cinf}{c_\infty}
\def\R{\mathbb{R}}

\def\N{\mathbb{N}}
\def\eps{\varepsilon}
\def\L{{\mathcal L}}

\usepackage{color}

\makeatletter\renewcommand{\subjclass}[2][2010]{\let\@oldtitle\@title\gdef\@title{\@oldtitle\footnotetext{#1 \emph{M\lowercase{athematics subject classification.}} #2}}}

\begin{document}

\title[Asymptotic estimates for the Keller-Segel model]{Asymptotic estimates for the parabolic-elliptic Keller-Segel model in the plane}

\author[J.F.~Campos Serrano\and J. Dolbeault]{Juan F.~Campos Serrano\and Jean Dolbeault}

\address[J.~Dolbeault]
{Ceremade (UMR CNRS no. 7534), Universit\'e Paris-Dauphine, Place de Lattre de Tassigny, F-75775 Paris C\'edex 16, France}
\email{dolbeaul@ceremade.dauphine.fr}
\address[J.F.~Campos]
{Ceremade (UMR CNRS no. 7534), Universit\'e Paris-Dauphine, Place de Lattre de Tassigny, F-75775 Paris C\'edex 16, France \& Departamento de Ingenier\'{\i}a Matem\'atica and CMM, Universidad de Chile, Casilla 170 Correo 3, Santiago, Chile}
\email{campos@ceremade.dauphine.fr, juanfcampos@gmail.com}

\begin{abstract} We investigate the large-time behavior of the solutions of the two-dimensional Keller-Segel system in self-similar variables, when the total mass is subcritical, that is less than $8\,\pi$ after a proper adimensionalization. It was known from previous works that all solutions converge to stationary solutions, with exponential rate when the mass is small. Here we remove this restriction and show that the rate of convergence measured in relative entropy is exponential for any mass in the subcritical range, and independent of the mass. The proof relies on symmetrization techniques, which are adapted from a paper of J.I.~Diaz, T.~Nagai, and J.-M.~Rakotoson, and allow us to establish uniform estimates for $L^p$ norms of the solution. Exponential convergence is obtained by the mean of a linearization in a space which is defined consistently with relative entropy estimates and in which the linearized evolution operator is self-adjoint. The core of proof relies on several new spectral gap estimates which are of independent interest.\end{abstract}

\maketitle
\thispagestyle{empty}

\section{Introduction}
Consider the two-dimensional parabolic-elliptic Keller-Segel system
\beq
\left\{\begin{array}{lcr}
\dr{u}{t} = \Delta u-\nabla\cdot(u\,\nabla v)& x\in\R^2\,, &t>0\;,\\[6pt]
v = G_2\ast u & x\in \R^2\,, &t>0\;,\\[6pt]
u(0,x) = n_0\geq 0 & x\in \R^2\,,&
\end{array}\right.
\label{KS}
\eeq
where $G_2$ denotes the Green function associated to $-\Delta$ on $\R^2$:
\[
G_2(x):=-\frac 1{2\pi}\,\log|x|\;,\quad x\in\R^2\,.
\]
The equation for the \emph{mass density} $u$ is parabolic, while the \emph{chemo-attractant density} $v$ solves an (elliptic) Poisson equation: $-\Delta v=u$. The drift term corresponds to an attractive mean-field nonlinearity, which has attracted lots of attention in mathematical biology in the recent years: see \cite{MR2448428,MR2013508,MR2073515, MR2270822,MR2430318,MR2430317} for some recent overviews. According to \cite{MR1046835,MR2103197,MR2226917,CarrilloChenLiuWang}, it is known that if
\be{cond}
n_0\in L^1_+\(\R^2\,,(1+\abs{x}^2)\,dx\)\;,\quad n_0\abs{\log n_0}\in L^1(\R^2)\quad\textrm{and}
\quad M:=\ix{n_0}< 8\,\pi\;,
\ee
then there exists a solution $u$, in the sense of distributions, that is global in time and such that $M=\ix{u(t,x)}$ is conserved along the evolution in the euclidean space $\R^2$. There is no non-trivial stationary solution of \eqref{KS} and any solution converges to zero locally as time gets large. In order to study the asymptotic behavior of $u$, it is convenient to work in self-similar variables. We define $R(t):=\sqrt{1+2\,t}$, $\tau(t):=\log R(t)$, and the rescaled functions $n$ and $c$ by
\[
u(t,x):=R^{-2}(t)\,n\(\tau(t),R^{-1}(t)\,x\)\quad\textrm{and}\quad v(t,x):=c\(\tau(t),R^{-1}(t)\,x\)\;.
\]
This time-dependent rescaling is the one of the heat equation. We observe that the non-linear term is also invariant under such a rescaling. The rescaled systems reads
\beq
\left\{\begin{array}{lcr}
\dr{n}{t} = \Delta n+\nabla\cdot(n\,x)-\nabla\cdot(n\,\nabla c)& x\in\R^2\,, &t>0\;,\\[6pt]
c = G_2 \ast n & x\in \R^2\,, &t>0\;,\\[6pt]
n(0,x) = n_0\geq 0 & x\in \R^2\,.&
\end{array}\right.
\label{RKS}
\eeq

Under Assumptions \eqref{cond}, it has been shown in \cite[Theorem 1.2]{MR2226917} that there exists a solution $n\in C^0(0,\infty;L^1(\R^2))\cap L^\infty_{\rm loc}(0,\infty;L^p(\R^2))$ for any $p\in(1,\infty)$ (also see \cite{MR2103197} for \emph{a priori} estimates) such that 
\[
\lim_{t\to\infty}\nrm{n(t,\cdot)-\ninf}1=0\quad\textrm{and}\quad\lim_{t\to\infty}\nrm{\nabla c(t,\cdot)-\nabla \cinf}2=0\;,
\]
where $(\ninf, \cinf)$ solves
\be{StatKS}
\ninf=M\,\frac{e^{\cinf-\abs{x}^2/2}}{\ix{e^{\cinf-\abs{x}^2/2}}}\quad\textrm{with}\quad \cinf=G_2 \ast \ninf\;.
\ee
Moreover, $\ninf$ is smooth and radially symmetric. Existence of a solution to \eqref{StatKS} has been established in \cite{MR1341221} by ordinary differential equation techniques and in \cite{MR2057636} by partial differential equation methods. The uniqueness has been shown in \cite{MR2249579}. To recall the dependence of $\ninf$ in~$M$, we will write it as $n_{\infty,M}$ whenever needed. 

A simple computation of the second moment shows that smooth solutions with mass larger than $8\,\pi$ blow-up in finite time; see for instance \cite{MR1046835}. The case $M=8\,\pi$ has been extensively studied. We shall refer to \cite{MR2249579,Blanchet20122142,MR2436186} for some recent papers on this topic. The asymptotic regime is of a very different nature in such a critical case. In the present paper, we shall restrict our purpose to the sub-critical case $M<8\,\pi$. 

In \cite{Blanchet2010533} (also see \cite{MR2103197}) it has been proved that there exists a positive mass $M_{\star}\le 8\,\pi$ such that for any initial data $n_0\in L^2(\ninf^{-1}\,dx)$ of mass $M<M_{\star}$ satisfying \eqref{cond}, System \eqref{RKS} has a unique solution $n\in C^0(0,\infty;L^1(\R^2))\cap L^\infty_{\rm loc}(0,\infty;L^p(\R^2))$ for any $p\in(1,\infty)$ such that
\[\label{Ineq:ExpRate}
\int_{\R^2}\abs{n(t,x)-\ninf(x)}^2\frac{dx}{\ninf(x)}\leq C\,e^{-\,\delta\,t}\quad\forall\;t\ge 0
\]
for some positive constants $C$ and $\delta$. Moreover $\delta$ can be taken arbitrarily close to $1$ as $M\to 0$. If $M<8\,\pi$, we may notice that the condition $n_0\in L^2(\ninf^{-1}\,dx)$ is stronger than~\eqref{cond}. Our main result is that $M_{\star}=8\,\pi$ and $\delta\ge1$, at least for a large subclass of solutions with initial datum $n_0$ satisfying the following technical assumption:
\be{TechnicalAssumption}
\exists\;\eps\in(0,8\,\pi-M)\quad\mbox{such that}\quad\int_0^sn_{0,*}(\sigma)\;d\sigma\leq\int_{B\(0,\sqrt{s/\pi}\)}n_{\infty,M+\eps}(x)\;dx\quad\forall\;s\ge0\;.
\ee
Here $n_{0,*}(\sigma)$ stands for the symmetrized function associated to $n_0$. Details will be given in Section~\ref{Sec:Symmetrization}.
\begin{Teo}\label{Thm:Main} Assume that $n_0$ satisfies \eqref{TechnicalAssumption},
\[
n_0\in L^2_+(\ninf^{-1}\,dx)\quad\textrm{and}\quad M:=\ix{n_0}< 8\,\pi\;.
\]
Then any solution of \eqref{RKS} with initial datum $n_0$ is such that
\[
\int_{\R^2}\abs{n(t,x)-\ninf(x)}^2\frac{dx}{\ninf(x)}\leq C\,e^{-\,2\,t}\quad\forall\;t\ge 0
\]
for some positive constant $C$, where $\ninf$ is the unique stationary solution to \eqref{StatKS} with mass~$M$.\end{Teo}
This result is consistent with the recent results of \cite{Calvez:2010fk} for the two-dimensional radial model and its one-dimensional counterpart (see Proposition~\ref{Prop:RadialGap} for more comments). For completeness, let us mention that results of exponential convergence for problems with mean field have been obtained earlier in \cite{MR2053570,MR2209130}, but only for interaction potentials involving much smoother kernels than $G_2$. The technical restriction~\eqref{TechnicalAssumption} on the initial datum looks rather strong as it imposes a uniform bound with a decay at infinity which is the one of a stationary solution (with larger mass). However, it is probably not as restrictive as it looks for well behaved solutions. For instance, it can reasonably be expected that the property holds true after a certain time delay for any solution corresponding to a smooth compactly supported initial datum. Solutions with a slower decay at infinity may be more difficult to handle. Both issues are open and probably require a significant effort for reaching a complete answer. This is why they are out of the scope of the present paper whose aim is to establish that the exponential convergence holds for arbitrary masses less than $8\pi$ and at a rate which is independent of the mass.

\medskip Our paper is organized as follows. In Section~\ref{Sec:Symmetrization}, we will apply symmetrization techniques as in \cite{MR1361010,MR1620286} to establish uniform estimates on $\nrm np$. Then we will prove the uniform convergence of~$n$ to $\ninf$ using Duhamel's formula: see Corollary~\ref{Cor:UniformCV} in Section~\ref{Sec:Uniform}. Section~\ref{Sec:Linearization} is devoted to the linearization of the problem around~$\ninf$ and to the study of the spectral gap of the linearized operator. A strict positivity result for the linearized entropy is also needed and will be proved in Section~\ref{Sec:Strict}. The proof of Theorem~\ref{Thm:Main} is completed in the last section. It is based on two estimates: Theorems~\ref{Teo:Gap} and~\ref{Teo:LimHLS} (also see Corollary~\ref{Cor:LimHLS}) that are established in Sections~\ref{Sec:Linearization} and~\ref{Sec:Strict} respectively. Some of the results of Sections~\ref{Sec:Linearization} and~\ref{Sec:Strict} (see Theorem~\ref{Teo:Gap} and Corollary~\ref{Cor:LimHLS}) have been announced without proof in~\cite{CD2012Cras} in connection with a new Onofri type inequality, whose linearized form is given in Inequality~\eqref{Ineq:ImprovedGapBis}.

\section{Symmetrization}\label{Sec:Symmetrization}

In this section, we adapt the results of J.I.~Diaz, T.~Nagai, and J.-M.~Rakotoson in \cite{MR1620286} to the setting of self-similar variables. Several key estimates are based on an earlier work of J.I.~Diaz and T.~Nagai for the bounded domain case: see \cite{MR1361010}. For a general introduction to symmetrization techniques, we refer to \cite{MR2001i:00001}. We shall therefore only sketch the main steps of the method and emphasize the necessary changes. 

\medskip To any measurable function $u:\R^2\mapsto [0,+\infty)$, we associate the distribution function defined by $\mu(t,\tau):=\abs{\{u>\tau\}}$ and its decreasing rearrangement given by
\[
u_*: [0,+\infty)\;\to\;[0,+\infty]\;,\quad s\;\mapsto\;u_*(s)=\inf\{\tau\geq0\;:\;\mu(t,\tau)\leq s\}\;.
\]
We adopt the following convention: for any time-dependent function $u:(0,+\infty)\times\R^2\to[0,+\infty)$, we will also denote by $u_*$ the decreasing rearrangement of $u$ with respect to the spatial variable, that is, $u_*(t,s)=u(t,.)_*(s)$.

Rearrangement techniques are a standard tool in the study of partial differential equations: in the framework of chemotaxis, see for instance \cite{MR572958,MR853732,MR733257} in case of bounded domains, and \cite{MR1620286} for unbounded domains. Let us briefly recall some properties of the decreasing rearrangement:
\begin{enumerate}
\item For every measurable function $F:\R^+\mapsto \R^+$, we have
\[
\ix{F(u)}=\ir{F(u_*)}\;.
\]
In particular, if $u\in L^p(\R^N)$ with $1\leq p\leq\infty$, then $\norm{u}_{L^p(\R^N)}=\norm{u_*}_{L^p(\R^N)}$.
\item If $u\in W^{1,q}(0,T; L^p(\R^N))$ is a nonnegative function, with $1\leq p<\infty$ and $1\leq q\leq \infty$, then $u_*\in W^{1,q}(0,T; L^p(0,\infty))$ and the formula
\[\label{dnr}
\int_0^{\mu(t,\tau)}\dr{u_*}{t}(t,\sigma)\;d\sigma=\int_{\{u(t,\cdot)>\tau\}}\dr{u}{t}(t,x)\;dx
\]
holds for almost every $t\in(0,T)$. Here $\mu(t,\tau)$ denotes $\abs{\{u(t,\cdot)>\tau\}}$. See \cite[Theorem~2.2, (ii), p.~167]{MR1620286} for a statement and a proof.
\end{enumerate}

As in \cite{MR1620286}, let us consider a solution $(n,c)$ of \eqref{RKS} and define
\[
k(t,s):=\int_0^s n_*(t,\sigma)\;d\sigma
\]
The strategy consists in finding a differential inequality for $k(t,s)$. Then, using a comparison principle, we will obtain an upper bound on the $L^p$ norm of $n$. In \cite{MR1620286}, the method was applied to~\eqref{KS}. Here we adapt it to the solution in rescaled variables, that is~\eqref{RKS}.
\begin{Lemma}\label{Lem:CompSym} If $n$ is a solution of \eqref{RKS} with initial datum $n_0$ satisfying the assumptions of Theorem~\ref{Thm:Main}, then the function $k(t,s)$ satisfies
\[
k\in L^\infty\parred{[0,+\infty)\times(0,+\infty)}\cap H^1\parred{[0,+\infty);W_{\rm{loc}}^{1,p}(0,+\infty)}\cap L^2\parred{[0,+\infty);W_{\rm{loc}}^{2,p}(0,+\infty)}
\]
and
\ben
\left\{\begin{array}{ll}
\dr{k}{t}-4\,\pi\,s\,\ddr{k}{s}-(k+2\,s)\,\dr{k}{s}\leq 0 & \textrm{a.e. in } (0,+\infty)\times(0,+\infty)\;,\\[6pt]
k(t,0)=0\;,\quad k(t,+\infty)=\ix{n_0} & \textrm{for } t\in(0,+\infty)\;,\\[6pt]
k(0,s)=\int_0^s (n_0)_*\;d\sigma& \textrm{for } s\geq0\;.
\end{array}\right.
\label{ineq}
\een
\end{Lemma}
\begin{proof} The proof follows the method of \cite[Proposition 3.1]{MR1620286}. We will therefore only sketch the main steps that are needed to adapt the results to the setting of self-similar variables and refer to \cite{MR1620286} for all technical details.

For $\tau\in(0, n_*(t,0))$ and $h>0$, define the truncation function $T_{\tau,h}$ on $(-\infty,+\infty)$ by
\ben
T_{\tau,h}&=& \left\{\begin{array}{ll}
0 & \textrm{if }s\leq\tau\\[6pt]
s-\tau & \textrm{if }\tau<s\leq \tau+h\\[6pt]
h & \textrm{if }\tau+h<s
\end{array}\right.
\een
and observe that $T_{\tau,h}(n(t,\cdot))$ belongs to $W^{1,p}(\R^2)$ since $n(t,\cdot)\in W^{1,p}(\R^2)$ and $T_{\tau,h}$ is Lipschitz continuous. Now we integrate \eqref{RKS} against $T_{\tau,h}(n)$ over $\R^2$, and integrate by parts to obtain
\[
\ix{\dr{n}{t}\,T_{\tau,h}(n)}+\ix{\nabla n\,\nabla T_{\tau,h}(n)}=\ix{n\,(\nabla c-x)\,\nabla T_{\tau,h}(n)}\;.
\]
We have that $\abs{\{n=\tau\}}=0$ for almost every $\tau\ge0$. Hence one can prove that
\[
\lim_{h\to 0}\frac1h\ix{\dr{n}{t}\,T_{\tau,h}(n)}=\int_{\{n>\tau\}}\dr{n}{t}(t,x)\;dx=\dr{k}{t}(t,\mu(t,\tau))\;.
\]
Next we observe that
\ben
\lim_{h\to 0}\frac1h \ix{\nabla n\,\nabla T_{\tau,h}(n)} &=& \lim_{h\to 0}\frac1h \parred{\int_{\{n>\tau\}}\abs{\nabla n}^2\,dx-\int_{\{n>\tau+h\}}\abs{\nabla n}^2\,dx}\\[6pt]
&=&\dr{}{\tau}\int_{\{n>\tau\}}\abs{\nabla n}^2\,dx\;.
\een
Consider the function
\[
\Phi_{\tau, h} = \int_0^s \sigma\,\dr{T_{\tau,h}}{\sigma}(\sigma)\;d\sigma = \left\{\begin{array}{ll}
0 & \textrm{if }s\leq\tau\;,\\[6pt]
\frac12\,(s^2-\tau^2) & \textrm{if }\tau<s\leq \tau+h\;,\\[6pt]
h(\tau + \frac h2) & \textrm{if }\tau+h<s\;.
\end{array}\right.
\]
Integrating the Poisson equation for $c$ against $\Phi_{\tau, h}(n)$, we get
\[
\ix{\nabla c\,\nabla \Phi_{\tau, h}(n)}= \ix{n\,\nabla c\,\nabla T_{\tau, h}(n)} = \ix{n\,\Phi_{\tau, h}(n)}\;,
\]
thus proving that
\ben
&&\lim_{h\to 0_+}\frac1h\ix{n\,\nabla c\,\nabla T_{\tau, h}(n)}\\[6pt]
&&=\lim_{h\to 0_+}\(\frac1{2\,h}\int_{\{\tau<n\leq\tau+h\}}n\,(n^2-\tau^2)\;dx+\int_{\{n >\tau+h\}}n\(\tau+\tfrac h2\)\,dx\)\\[6pt]
&&=\tau\int_{\{n>\tau\}}n\;dx =\dr{k}{s}(t,\mu(t,\tau))\,k(t,\mu(t,\tau))
\een
since $\tau=n_*(t,\mu(t,\tau))=\dr{k}{s}(t,\mu(t,\tau))$ and $\int_{\{n>\tau\}}n\;dx = \int_0^{\mu(t,\tau)}n_*(t,s)\;ds=k(t,\mu(t,\tau))$. On the other hand,
\ben
\lim_{h\to 0_+}\frac1h\ix{n(x)\,x\cdot\nabla T_{\tau,h}(x)} & = & \lim_{h\to 0_+}\frac1h\ix{x\cdot\nabla \Phi_{\tau,h}(n)}=-\lim_{h\to 0_+}\frac2h\ix{\Phi_{\tau,h}(n)}\\[6pt]
&=& -\,2\,\tau\,|\{n>\tau\}|\, =\, -\,2\,\dr{k}{s}(t,\mu(t,\tau))\,\mu(t,\tau)\;.
\een
Using the inequality
\[
4\,\pi\,\mu(t,\tau)\leq \frac{\partial\mu}{\partial\tau}(t,\tau)\,\dr{}{\tau}\int_{\{n>\tau\}}\abs{\nabla n}^2\,dx\,,
\]
(see \cite[Proof of Lemma 4, p.~669]{MR1361010}, and also \cite[pp.~25-26]{MR733257} or \cite[p.~20]{MR853732}, and \cite{MR551065} for an earlier reference) we obtain
\[
1\leq -\frac{\frac{\partial\mu}{\partial\tau}(t,\tau)}{4\,\pi\,\mu(t,\tau)}\,\parred{-\dr{k}{t}(t,\mu(t,\tau))+\dr{k}{s}(t,\mu(t,\tau)) \big(k(t,\mu(t,\tau))+2\,\mu(t,\tau)\big)}
\]
for almost every $\tau\in(0,n_*(t,0))$. Integrating over $(\tau_1,\tau_2)\subset (0,n_*(t,0))$, as in \cite[Lemma 2]{MR1172451}, we get
\ben
\frac1{4\,\pi} \int_{\mu(t,\tau_1)}^{\mu(t,\tau_2)}\parred{-\dr{k}{t}(t,s)+\dr{k}{s}(t,s) \big(k(t,s)+2\,s\big)}\,\frac{ds}s &\le&\tau_1-\tau_2
\een
where
\ben
\tau_1-\tau_2 &=&\dr{k}{s}(t,\mu(t,\tau_1))-\dr{k}{s}(t,\mu(t,\tau_2))\;.
\een
Hence dividing by $(\mu(t,\tau_2)-\mu(t,\tau_2))$ and then taking the limit completes the proof. 
\end{proof}

The next result is adapted from \cite[Proposition A.1, p.~676]{MR1361010} and \cite[Proposition 3.2, p.~172]{MR1620286}. Although it is unnecessarily general for our purpose, as the function $g$ below is extremely well defined (and independent of $t$), we keep it as in J.I.~Diaz \emph{et al.} and give a sketch of the proof, for completeness.
\begin{Prop}\label{Prop:comparison}
Let $f$, $g$ be two continuous functions on $Q=\R^+\times(0,+\infty)$ such that
\begin{enumerate}
\item[{\rm (i)}] $f$, $g\in L^\infty (Q)\cap L^2(0,+\infty; W^{2,2}_{\rm{loc}}(0,+\infty))$, $\dr{f}{t}$, $\dr{g}{t}$ $\in L^2(0,+\infty; L^{2}_{\rm{loc}}(0,+\infty))$,
\item[{\rm (ii)}] $\abs{\dr{f}{s}(t,s)}\leq C(t)$ and $\abs{\dr{g}{s}(t,s)}\leq C(t)\max\{s^{-1/2},1\}$, for some continuous function $t\mapsto C(t)$ on $\R^+$.
\end{enumerate}
If $f$ and $g$ satisfy
\[
\left\{\begin{array}{c}
\dr{f}{t}-4\,\pi\,s\,\ddr{f}{s}-(f+2\,s)\,\dr{f}{s}\leq \dr{g}{t}-4\,\pi\,s\,\ddr{g}{s}-(g+2\,s)\,\dr{g}{s} \textrm{ a.e. in }Q\;,\\[6pt]
f(t,0)=0=g(t,0)\quad\textrm{and}\quad f(t,+\infty)\leq g(t,+\infty) \textrm{ for any } t\in(0,+\infty)\;,\\[6pt]
f(0,s)\leq g(0,s) \textrm{ for } s\geq0\;, \textrm{and } g(t,s)\geq 0 \textrm{ in }Q\;,
\end{array}\right.
\]
then $f\leq g$ on $Q$.
\end{Prop}
\begin{proof} Take $w=f-g$. We have
\[
\dr{w}{t}-4\,\pi\,s\,\ddr{w}{s}-\,2\,s\,\dr{w}{s}\leq w\,\dr{f}{s}+g\,\dr{w}{s}\,.
\]
Multiplying by $w_+/s$, and integrating over $(\delta, L)$ with $0<\delta<1<L$, we obtain
\begin{multline*}
\frac12\dr{}{t}\int_\delta^L \frac{w^2_+}s\;ds + 4\,\pi\int_\delta^L\parred{\dr{w_+}{s}}^2ds - \parcuad{4\,\pi\,\dr{w_+}{s}(t,s)\,w_+(t,s)}_{s=\delta}^{s=L}-\int_\delta^L \dr{}{s}\parred{w_+^2}\;ds\\
\leq \int_\delta^L\parred{w_+^2\,\dr{f}{s}+w_+\,\dr{w}{s}\,g}\,\frac{ds}s
\end{multline*}
thus showing that
\[
\frac12\,\dr{}{t}\int_\delta^L \frac{w^2_+}s\;ds + 4\,\pi\int_\delta^L\parred{\dr{w_+}{s}}^2ds \leq C(t) \int_\delta^L \frac{w^2_+}s\;ds+\int_\delta^L \frac{w_+}s\,\dr{w}{s}\,g\;ds+G(t,\delta,L)\;,
\]
where $G(t,\delta,L)$ uniformly (with respect to $t\ge 0$) converges to $0$ as $\delta\to 0$ and $L\to +\infty$. Now using the fact that $g(t,s)/\sqrt s\leq C(t)$ we obtain that, for some constant $K>0$,
\[
\int_\delta^L \frac{w_+}s\,\dr{w}{s}\,g\;ds \leq 4\,\pi\int_\delta^L \parred{\dr{w_+}{s}}^2ds+K\,C^2(t)\int_\delta^L \frac{w^2_+}s\;ds\;,
\]
yielding
\[
\frac12\dr{}{t}\int_\delta^L \frac{w^2_+}s\;ds \leq (1+K\,C(t))\,C(t) \int_\delta^L \frac{w^2_+}s\;ds+G(t,\delta,L)\;.
\]
From Gronwall's lemma and $w_+(0,s)=0$, with $R(t):=2\int_0^t(1+K\,C(\tau))\,C(\tau)\,d\tau$, it follows that
\[
\int_\delta^L \frac{w^2_+}s\;ds \leq 2\,e^{R(t)}\int_0^te^{-R(\tau)}\,G(\tau,\delta,L)\;d\tau\;.
\]Taking the limit as $\delta \to 0$ and $L\to+\infty$, we obtain
\[
\int_0^\infty \frac{w^2_+}s\;ds \leq 0\;,
\]
which implies $f\leq g$ and concludes the proof.
\end{proof}

Using Lemma~\ref{Lem:CompSym} and Proposition~\ref{Prop:comparison}, we can now establish uniforms bounds on $\norm{n}_{L^p(\R^2)}$ and $\norm{\nabla c}_{L^\infty(\R^2)}$, which are enough to justify all integrations by parts that are needed in this paper.

\begin{Teo}\label{Thm:UniformBound} Assume that $n_0\in L^2_+(\ninf^{-1}\,dx)$ satisfies \eqref{TechnicalAssumption} and $ M:=\ix{n_0}< 8\,\pi$. Then, for any $p\in[1,\infty]$, there exist positive constants $C_1=C_1(M,p)$ and $C_2=C_2(M,p)$ such that
\[
\norm{n(t,\cdot)}_{L^p(\R^2)}\leq C_1\quad\mbox{and}\quad \norm{\nabla c(t,\cdot)}_{L^\infty(\R^2)}\leq C_2\quad\forall\;t>0\;.
\]
\end{Teo}
\begin{proof} The function $M_\eps(s):=\int_{B(0,\sqrt{s/\pi})}n_{\infty,M+\eps}\;dx$ satisfies
\[
4\,\pi\,s\,M''_\eps+2\,s\,M_\eps'+M_\eps\,M_\eps'=0\;.
\]
By direct application of Proposition~\ref{Prop:comparison}, we obtain
\[
k(t,s)\leq M_\eps(s)\quad\forall\;(t,s)\in Q\;.
\]
By \cite[Lemma 1.33]{MR853732}, we deduce
\[
\norm{n_*}_{L^p(0,\infty)}\leq \norm{M_{\eps}'}_{L^p(0,\infty)}\;,
\]
which yields the result. More details on $M_\eps$ and \emph{cumulated densities} will be given in Section~\ref{Sec:SpectralGap}.
\end{proof}

\section{Uniform convergence}\label{Sec:Uniform}

Wit the boundedness results of Section~\ref{Sec:Symmetrization} in hands, we can now prove a result of uniform convergence for $n$ and $\nabla c$, if $(n,c)$ is given as a solution of \eqref{RKS} satisfying the assumptions of Theorem~\ref{Thm:Main}.

\medskip Consider the kernel associated to the Fokker-Planck equation
\[
K(t,x,y):=\frac 1{2\pi\,(1-e^{-\,2\,t})}\,e^{-\frac 12\,\frac{|x-e^{-t}y|^2}{1-e^{-\,2\,t}}}
\quad x\in\R^2\,,\quad y\in\R^2\,,\quad t>0\;.
\]
This definition deserves some explanations. If $n$ is a solution of
\[
\dr{n}{t} = \Delta n+\nabla\cdot(n\,x)
\]
with initial datum $n_0$, then $u(\tau,\xi)=R^{-\,2}\,n\(\log R,R^{-1}\,\xi\)$ with $R=R(\tau)=\sqrt{1+2\,\tau}$ is a solution of the heat equation
\[
\dr{u}{\tau} = \Delta u\;,\quad u(\tau=0,\cdot)=n_0\;,
\]
whose solution is given by
\[
u(\xi,\tau)=\frac 1{4\,\pi\,\tau}\int_{\R^2}e^{-\frac{|\xi-y|^2}{4\tau}}\,n_0(y)\;dy\;.
\]
By undoing the change of variables, we get that the solution of the Fokker-Planck equation is given by
\[
n(t,x)=\int_{\R^2}K(t,x,y)\,n_0(y)\;dy\;.
\]
Consider now a solution of \eqref{RKS}. We have the following Duhamel formula.
\begin{Lemma}\label{Lem:Duhamel} Assume that $n$ is a solution of \eqref{RKS} with initial data satisfying \eqref{cond}. Then for any $t>0$, $x\in\R^2$, we have
\[
n(t,x)=\int_{\R^2}K(t,x,y)\,n_0(y)\;dy+\int_0^t\int_{\R^2}\nabla_xK(t-s,x,y)\cdot n(s,y)\,\nabla c(s,y)\;dy\;ds\;.
\]
\end{Lemma}
This is a standard fact whose proof relies on the fact that $(t,x)\mapsto K(t,x,y)$ is a solution of the Fokker-Planck equation with a $\delta$-Dirac function initial value. Details are left to the reader.

\medskip Using the semi-group property, we deduce from Lemma~\ref{Lem:Duhamel} the expression for $n(t+1,x)$ in terms of $n(t,x)$ as
\[
n(t+1,x)=\int_{\R^2}K(t,x,y)\,n(t,y)\;dy+\int_t^{t+1}\int_{\R^2}\nabla_xK(t+1-s,x,y)\cdot n(s,y)\,\nabla c(s,y)\;dy\;ds
\]
for any $t\ge 0$. Since $\ninf$ is a stationary solution, we can also write that
\[
\ninf(x)=\int_{\R^2}K(t,x,y)\,\ninf(y)\;dy+\int_t^{t+1}\int_{\R^2}\nabla_xK(t+1-s,x,y)\cdot \ninf(y)\,\nabla \cinf(y)\;dy\;ds
\]
for any $t\ge 0$. By taking the difference of the two expressions, we get that
\begin{multline*}
n(t+1,x)-\ninf(x)=\int_{\R^2}K(t,x,y)\,(n(t,y)-\ninf(y))\;dy\\[6pt]
+\int_t^{t+1}\int_{\R^2}\nabla_xK(t+1-s,x,y)\cdot \(n(s,y)\,\nabla c(s,y)-\ninf(y)\,\nabla \cinf(y)\)\,dy\;.
\end{multline*}
This provides a straightforward estimate, which goes as follows:
\begin{multline*}
\nrm{n(t+1, \cdot)-\ninf}\infty\le\|K(t,\cdot,\cdot)\|_{L^\infty(\R^2_x;L^r(\R^2_y)}\,\nrm{n(t,\cdot)-\ninf}1\\[6pt]
+\int_0^1\|\nabla K(s,\cdot,\cdot)\|_{L^\infty(\R^2_x;L^r(\R^2_y))}\,ds\;\mathcal R(t)
\end{multline*}
where $\frac 1p+\frac 1q+\frac 1r=1$ with $p\in(2,\infty)$, $q\in[2,\infty)$ and $r\in(1,2)$, and
\[
\mathcal R(t):=\sup_{s\in(t,t+1)}\(\nrm{n(s,\cdot)}p\,\nrm{\nabla c(s,\cdot)-\nabla \cinf}q+\nrm{n(s,\cdot)-\ninf}p\,\nrm{\nabla \cinf}q\)\,.
\]
A direct computation shows that
\[
\nabla K(t,x,\cdot)=\frac{e^{-t}}{2\pi\,(1-e^{-\,2\,t})^2}\,(e^t\,x-y)\,e^{-\frac{|e^tx-y|^2}{2\,(e^{2t}-1)}}
\]
and hence, for some explicit, finite function $r\mapsto\kappa(r)$,
\[
\|\nabla K(t,x,\cdot)\|_{L^r(\R^2_y)}=\frac{e^{-t}}{2\pi\,(1-e^{-2t})^2}\(\int_{\R^2}|z|^r\,e^{-\frac{r\,|z|^2}{2\,(e^{2t}-1)}}\,dz\)^\frac 1r=\kappa(r)\,e^{3t}\,(e^{2t}-1)^\frac{2-3\,r}{2\,r}
\]
is integrable in $t\in(0,1)$ if $r\in[1,2)$. On the other hand, $\mathcal R(t)$ converges to $0$ by Theorem~\ref{Thm:UniformBound} and the fact that
\[
\lim_{t\to\infty}\nrm{n(t,\cdot)-\ninf}1=0\quad\textrm{and}\quad\lim_{t\to\infty}\nrm{\nabla c(t,\cdot)-\nabla\cinf}2=0
\]
according to \cite[Theorem 1.2]{MR2226917}. Hence we have shown the \emph{uniform convergence} of $n$ towards~$\ninf$ as $t\to\infty$ and
\[
\lim_{t\to\infty}\nrm{n(t,\cdot)-\ninf}p=0
\]
for any $p\in[1,\infty]$, by H\"older's interpolation. As for the convergence of $\nabla c(t,\cdot)$ towards $\nabla\cinf$ as $t\to\infty$ in $L^q(\R^2)$ for $q\in(2,\infty]$, we need one more interpolation inequality.
\begin{Lemma}\label{Lem:Linfty} If $h=(-\Delta)^{-1}\rho$ for some function $\rho\in L^{2-\varepsilon}\cap L^{2+\varepsilon}(\R^2)$, with $\varepsilon\in(0,1)$, then there exists an explicit positive constant $C=C(\varepsilon)$ such that
\[
\nrm{\nabla h}\infty\le C\,\(\nrm\rho{2-\varepsilon}+\nrm\rho{2+\varepsilon}\)\,.
\]
\end{Lemma}
\begin{proof} This follows by a direct computation. We can estimate $|\nabla h|$ by
\[
|\nabla h(x)|=\frac 1{2\pi}\int_{\R^2}\frac{\rho(y)}{|x-y|}\;dy
\]
for any $x\in\R^2$ and split the integral into two pieces corresponding to $|x-y|<1$ and $|x-y|\ge1$: by H\"older's inequality, we obtain that
\[
\frac 1{2\pi}\int_{|x-y|<1}\frac{\rho(y)}{|x-y|}\;dy\le C_1(\varepsilon)\,\nrm\rho{2+\varepsilon}
\]
with $C_1(\varepsilon)=\frac 1{2\pi}\(2\pi\,(1+\varepsilon)/\varepsilon\)^{(1+\varepsilon)/(2+\varepsilon)}$ and
\[
\frac 1{2\pi}\int_{|x-y|\ge 1}\frac{\rho(y)}{|x-y|}\;dy\le C_2(\varepsilon)\,\nrm\rho{2-\varepsilon}
\]
with $C_2(\varepsilon)=\frac 1{2\pi}\(2\pi\,(1-\varepsilon)/\varepsilon\)^{(1-\varepsilon)/(2-\varepsilon)}$. The conclusion holds with $C=\max_{i=1,2}C_i$.\end{proof}

Hence we have also shown the \emph{uniform convergence} of $\nabla c$ towards $\nabla\cinf$ as $t\to\infty$. By H\"older's interpolation, the convergence holds in $L^q(\R^2)$ for any $q\in[2,\infty]$. Summarizing all results of this section, we have shown the following limits.
\begin{Corollary}\label{Cor:UniformCV} Assume that $n$ is a solution of \eqref{RKS} with initial data satisfying the assumptions of Theorem~\ref{Thm:Main}. Then 
\[
\lim_{t\to\infty}\nrm{n(t,\cdot)-\ninf}p=0\quad\textrm{and}\quad\lim_{t\to\infty}\nrm{\nabla c(t,\cdot)-\nabla\cinf}q=0
\]
for any $p\in[1,\infty]$ and any $q\in[2,\infty]$.
\end{Corollary}

\section{Spectral gap of the linearized operator {\texorpdfstring{$\mathcal L$}{L}}}\label{Sec:Linearization}

Assume that $n$ is a solution of \eqref{RKS} and consider $f$ and $g$ defined for any $(t,x)\in\R^+\times\R^2$~by
\[
n(t,x)=\ninf(x)\,(1+f(t,x))\quad\textrm{and}\quad c(t,x)=\cinf(x)\,(1+g(t,x))\;.
\]
Then $(f,g)$ is a solution of the nonlinear problem
\beq
\left\{\begin{array}{lcr}
\dr{f}{t} -\L\,f = -\frac1{\ninf} \nabla\cdot\parcuad{f\,\ninf\,\nabla(g\,\cinf)}& x\in\R^2\,, &t>0\;,\\[6pt]
-\Delta (g\,\cinf) = f\,\ninf & x\in \R^2\,, &t>0\;,
\end{array}\right.
\label{pblin}
\eeq
where $\L$ is the linear operator
\[
\L\,f=\frac1{\ninf}\nabla\cdot\parcuad{\ninf\nabla(f-g\,\cinf)}.
\]
Such a formulation has already been used in \cite{Blanchet2010533}, but there only estimates of the lowest eigenvalues of $\L$ were considered. Here we are going to establish the exact value of the gap. The goal of this section is indeed to establish that $\mathcal L$ has a spectral gap in an appropriate functional setting. To characterize the spectrum of $\mathcal L$, it is necessary to specify the domain of the operator $\mathcal L$. Heuristically, it is simpler to identify the eigenfunctions corresponding to the lowest eigenvalues and define only afterwards the norm for which $\mathcal L$ turns out to be self-adjoint. We will
\\- identify some eigenfunctions of the linearized Keller-Segel operator $\mathcal L$ in Section~\ref{Sec:SomeEigenfunctions},
\\- characterize the kernel of $\mathcal L$ in Section~\ref{Sec:Kernel},
\\- determine an adapted functional setting for $\mathcal L$ and related operators in Section~\ref{Sec:FunctSetting},
\\- show that the spectrum of $\mathcal L$ is discrete in Section~\ref{Sec:Discrete},
\\- and finally establish a spectral gap inequality in Section~\ref{Sec:SpectralGap}.

\subsection{Some eigenfunctions of the linearized Keller-Segel operator {\texorpdfstring{$\mathcal L$}{L}}}\label{Sec:SomeEigenfunctions}

Using the fact that $\ninf$ depends on $x=(x_1,x_2)\in\R^2$ and on the mass parameter $M$, we observe that the functions
\ben
f_{0,0}&=&\partial_M\log n_{\infty,M}\;,\\[6pt]
f_{1,i}&=&\partial_{x_i}\log n_{\infty,M}\;,\quad i=1\,,\;2\;,\\[6pt]
f_{0,1}&=&1+\tfrac 12\,x\cdot\nabla\log n_{\infty,M}\;,
\een
are eigenfunctions of $\L$. Here $\partial_Mn_{\infty,M}$ denotes the derivative of the function $\ninf=n_{\infty,M}$ with respect to the mass parameter $M$, while $\partial_{x_i}$ stands for $\partial/\partial_{x_i}$. We shall use two indices for the numbering of the eigenfunctions because of a spherical harmonics decomposition that will be studied in Section~\ref{Sec:SpectralGap}. A precise statements goes as follows.
\begin{Lemma}\label{Lem:Eigenfunctions} With the above notations, we have
\ben
\L\,f_{0,0} &=&0\;,\\[6pt]
\L\,f_{1,i} &=&-f_{1,i}\;,\\[6pt]
\L\,f_{0,1} &=&-\,2\,f_{0,1}\;.
\een
\end{Lemma}
\begin{proof}
Assume that $M\in(0,8\,\pi)$ and consider the unique solution $\ninf$ of \eqref{StatKS}, which is also the unique stationary solution of \eqref{KS} such that \eqref{cond} holds. For brevity, we shall omit to mention the dependence of $\ninf=n_{\infty,M}$ in $M$.

\medskip Let us differentiate with respect to $M$ each term of $\Delta\,\ninf+\nabla\cdot(\ninf\,x)-\nabla\cdot(\ninf\nabla \cinf)=0$, where $\cinf=G_2*\ninf$. It is straightforward to check that $g_{0,0}:=\partial_M\log\cinf$ is such that $g_{0,0}\,\cinf=G_2\ast (f_{0,0}\,\ninf)$ and $\L\,f_{0,0}=0$. Since
\[
-\Delta\,\cinf=M\,\frac{e^{\cinf-\frac 12\,|x|^2}}{\ix{e^{\cinf-\frac 12\,|x|^2}}}=\ninf\;,
\]
it is clear that $g_{0,0}$ is non-trivial, and therefore $f_{0,0}=\partial_M\log \ninf$ is a non-trivial solution to $\L\,f=0$.

\medskip By computing
\[
0=\frac\partial{\partial x_1}\Big(\Delta \ninf+\nabla\cdot(x\,\ninf)-\nabla\cdot(\ninf\,\nabla \cinf)\Big)\quad\mbox{with}\quad -\Delta\(\frac{\partial \cinf}{\partial x_1}\)=\frac{\partial \ninf}{\partial x_1}
\]
and observing that
\[
\frac\partial{\partial x_1}\nabla\cdot(x\,\ninf)=\frac{\partial \ninf}{\partial x_1}+\nabla\cdot\(x\,\frac{\partial \ninf}{\partial x_1}\)\,,
\]
we obtain that $f_{1,1}:=\partial_{x_1}\log\ninf$ associated with $g_{1,1}=\frac 1{\cinf}\,\partial_{x_1}\cinf$ is an eigenfunction of $\L$, such that $-\L\,f_{1,1}=f_{1,1}$. The same observation holds if we differentiate with respect to $x_i$,~$i=2$.

\medskip Next consider the dilation operator $D:=x\cdot\nabla$. If $a$ is a vector valued function, an elementary computation shows that
\[
D\,(\nabla\cdot a)=\nabla\cdot(D\,a)-\nabla\cdot a\;.
\]
Since $a=\nabla\ninf+x\,\ninf-\ninf\,\nabla \cinf$ is such that $\nabla\cdot a=0$, we get $D\,(\nabla\cdot a)=\nabla\cdot(D\,a)$ and hence
\[
0=D\Big(\Delta \ninf+\nabla\cdot(x\,\ninf)-\nabla\cdot(\ninf\,\nabla \cinf)\Big)=\nabla\cdot D\,\Big(\nabla\ninf+x\,\ninf-\ninf\,\nabla \cinf\Big)\,.
\]
Next, we observe that
\[
D\,(\nabla\ninf)=\nabla\,(D\,\ninf)-\nabla\,\ninf
\]
so that
\[
\nabla\cdot D\,(\nabla\ninf)=\Delta(D\,\ninf)-\Delta\,\ninf\;.
\]
It is also straightforward to observe that
\[
D\,(x\,\ninf)=x\,\ninf+x\,D\,\ninf\quad\mbox{and}\quad D\,(\nabla \cinf)=\nabla\,(D\,\cinf)-\nabla\,\cinf\;.
\]
Let $f_{0,1}=1+\frac 12\,D\,\log \ninf=1+\frac 1{2\,\ninf}\,D\,\ninf$. By writing $D\,(\Delta \cinf+\ninf)=0$, we get
\[
-\Delta\(D\,\cinf\)+2\,\Delta \cinf=D\,\ninf=2\,(f_{0,1}-1)\,\ninf\;,
\]
since
\[
D\,(\Delta c)=\Delta(D\,c)-\,2\,\Delta c\;.
\]
Hence, using the fact that $2\,\Delta\cinf=-\,2\,\ninf$, the function $g_{0,1}:=\frac 1{\cinf}\,(-\Delta)^{-1}(\ninf\,f_{0,1})$ is given by
\[
\cinf\,g_{0,1}=\frac 12\,D\,\cinf\;.
\]
Collecting these identities, we have found that
\[
2\,\ninf\,\L\,(D\,\log \ninf)-\nabla\cdot\left[\nabla\ninf-x\,\ninf-\,2\,\ninf\Big(\nabla(\cinf\,g_{0,1})-\nabla \cinf\Big)+\ninf\,D\,(\nabla \cinf)\right]=0\;.
\]
Using
\[
2\Big(\nabla(\cinf\,g_{0,1})-\nabla \cinf\Big)=\nabla(D\,\cinf)-\,2\,\nabla \cinf=D\,(\nabla \cinf)-\nabla \cinf\;,
\]
this gives
\[
\ninf\,\L\,(D\,\log \ninf)-\Delta \ninf+\nabla\cdot(x\,\ninf-\ninf\,\nabla \cinf)=0\;.
\]
Hence, owing to the fact that $D\,\log \ninf=2\,(f_{0,1}-1)$ and
\[
\ninf\,\L\,(D\,\log \ninf)=2\,\ninf\,\mathcal L f_{0,1}+2\,\nabla\cdot(\ninf\nabla \cinf)\;,
\]
we get
\begin{multline*}
-\,2\,\ninf\,\L\,f_{0,1}=-\Delta \ninf+\nabla\cdot(x\,\ninf+\ninf\,\nabla \cinf)\\
=2\,\nabla\cdot(x\,\ninf)=4\,\ninf\,\(1+\frac{D\,\ninf}{2\,\ninf}\)=4\,\ninf\,f_{0,1}\;.
\end{multline*}
We have finally found that $-\L\,f_{0,1}=2\,f_{0,1}$, which completes the proof.\end{proof}

\begin{Rem} The fact that $1$ and $2$ are eigenvalues of $\L$ was known in the limit $M\to0_+$: see~\cite{Blanchet2010533}. It is remarkable that these two eigenvalues are independent of $M$ but this can be explained by noticing that the corresponding eigenfunctions are associated with invariances of the problem before rescaling.

The functions $\partial_{x_i}\log \ninf$, $i=1$, $2$ correspond to the invariance under translation in the directions~$x_i$. A decentered self-similar solution would converge in self-similar variables to the stationary solution, in relative entropy, exactly at a rate $e^{-t}$, thus showing that $\lambda_{1,1}=\lambda_{1,2}=1$ are eigenvalues by considering the asymptotic regime.

The function $D\,\log \ninf$ is associated with the scaling invariance. In original variables, a scaling factor corresponds to a translation in time at the level of the self-similar solution and it can easily be checked that, in self-similar variables, a solution corresponding to the stationary solution rescaled by a factor different from $1$ converges, in relative entropy, exactly at a rate $e^{-\,2\,t}$, thus showing that $\lambda_{0,1}=2$ is also an eigenvalue by considering the asymptotic regime.
\end{Rem}

\subsection{The kernel of the linearized Keller-Segel operator {\texorpdfstring{$\mathcal L$}{L}}}\label{Sec:Kernel}

By definition of $\ninf$, we know that $\log\ninf=\mu_0(M)+\cinf-\frac 12\,|x|^2$, so that $f_{0,0}=\mu_0'(M)+g_{0,0}\,\cinf$ where $g_{0,0}=\partial_M\log\cinf$ is such that $-\Delta(g_{0,0}\,\cinf)=-\Delta f_{0,0}=f_{0,0}\,\ninf$. The normalization constant $\mu_0$ is determined by the condition that $M=\ix{\ninf}$, that is $\mu_0=\log M-\log\big(\ix{e^{\cinf-|x|^2/2}}\big)$. By differentiating with respect to~$M$, we also get that
\[
\mu_0'(M)=\frac 1M\left[1-\ix{g_{0,0}\,\ninf\,\cinf}\right]\;.
\]
The function $f=f_{0,0}$ solves $\L\,f=0$ and is such that the equation for $g=f/\cinf$ reads
\[\label{linearizedP}
-\Delta\,f=\ninf\,f\;.
\]
It is not \emph{a priori} granted that such an equation has at most one solution, up to a multiplication by a constant. The uniqueness issue is the purpose of our next result.
\begin{Prop}\label{Prop:KerL} The kernel $\mathop{Ker}(\L)$ is generated by $f_{0,0}=\partial_M\log \ninf$, which is the unique solution in $L^2(\R^2,\ninf\,dx)$, up to a multiplication by a constant, to
\[
-\Delta f_{0,0}=f_{0,0}\,\ninf\;.
\]\end{Prop}
\begin{proof} We have already seen that $f_{0,0}\in\mathop{Ker}(\L)$. It remains to prove that $\mathop{Ker}(\L)$ is one-dimensional. Let $f$ be such that $\L\,f=0$ and $g=\cinf^{-1}\,G_2*(f\,\ninf)$. An elementary computation shows that
\[
0=\ix{\L\,f\,(f-g\,\cinf)\,\ninf}=\ix{|\nabla(f-g\,\cinf)|^2\,\ninf}\;,
\]
thus proving that $f=g\,\cinf+\mu_0'$ for some real constant $\mu_0'$ (depending eventually on $M$, with the same notations as above). Hence any solution of $\L\,f=0$ has to solve
\[
\mathcal H\,f=0
\]
where $\mathcal H:=-\Delta-\ninf$ is a Schr\"odinger operator with potential $\ninf$, at least if one assumes that $\nabla(f-G_2*(f\,\ninf))$ belongs to $L^2(\ninf\,dx)$. As we shall see later in the discussion of the domain of definition of $\L$, this is indeed a natural assumption. Altogether, we are interested in characterizing the ground state of the Schr\"odinger operator $\mathcal H$ (with energy level $0$) and prove that it is uniquely determined, up to a multiplication by a constant. It is clear that $\mathcal H$ has no negative eigenvalue, otherwise the \emph{free energy} functional
\[
F[n]:=\ix{n\,\log\(\frac n\ninf\)}\\+\frac 1{4\,\pi}\iint_{\R^2\times\R^2}(n(x)-\ninf(x))\,\log|x-y|\,(n(y)-\ninf(y))\;dx\,dy
\]
would not achieve its minimum for $n=\ninf$ (see \cite{MR2226917} for a proof). 

Since $\ninf$ is radially symmetric (see for instance \cite{MR2226917} for a summary of known results), Schwarz' symmetrization applied to $\mathcal H$ shows that the ground state is radially symmetric. The function $\ninf$ seen as a potential, is smooth. By standard elliptic theory, the ground state is smooth as well. Hence, if $f\in H^1(\R^2)$ solves $\mathcal H\,f=0$, it is uniquely determined as a solution of an ordinary differential equation by the Cauchy-Lipschitz theorem, up to a standard analysis at the origin. Indeed, by considering abusively $\ninf$ and $f$ as functions of $r=|x|$, we find that $f$ is given by
\begin{eqnarray*}
&&f''+\tfrac 1r\,f'+\ninf\,f=0\\
&&f(0)=1\;,\quad f'(0)=0
\end{eqnarray*}
(up to a multiplication by an arbitrary constant). This concludes the proof.
\end{proof}

\subsection{Functional setting and operators}\label{Sec:FunctSetting}

In order to go further in the spectral analysis, to define correctly the domain of the operator $\L$, to justify the assumption that $\nabla(f-G_2*(f\,\ninf))$ belongs to $L^2(\ninf\,dx)$ and to establish spectral gap estimates which are crucial for our analysis, some considerations on the functional setting are in order. On $L^2(\R^2,\ninf\,dx)$, let us consider the quadratic form $\mathsf Q_1$ obtained by linearization around $\ninf$, that is
\[
\mathsf Q_1[f]:=\lim_{\varepsilon\to0}\frac 1{\varepsilon ^2}\,F[\ninf(1+\varepsilon\,f)]\;.
\]
It takes the form
\[
\mathsf Q_1[f]=\ix{|f|^2\,\ninf}+\frac 1{2\pi}\iint_{\R^2\times\R^2}f(x)\,\ninf(x)\,\log|x-y|\,f(y)\,\ninf(y)\;dx\,dy\;.
\]
As a consequence of its definition, $\mathsf Q_1$ is nonnegative.
\begin{Lemma}\label{Lem:LimHLS} Assume that $M\in(0,8\,\pi)$ and consider $\ninf$ defined by \eqref{StatKS}. Then
\be{Ineq:LimHLS}
-\frac 1{2\pi}\iint_{\R^2\times\R^2}f(x)\,\ninf(x)\,\log|x-y|\,f(y)\,\ninf(y)\;dx\,dy=\ix{f\,\ninf\,g\,\cinf}\le\ix{|f|^2\,\ninf}
\ee
for any $f\in L^2(\R^2,\ninf\,dx)$, where $g\,\cinf=G_2*(f\,\ninf)$. Moreover, if
\be{Cdt:Average}
\ix{f\,f_{0,0}\,\ninf}=0\;,
\ee
then equality holds in the above inequality if and only if $f=0$.
\end{Lemma}
Notice that, if $f\in L^2(\R^2,\ninf\,dx)$ is such that
\be{Cdt:OrthogConst}
\ix{f\,\ninf}=0\;,
\ee
then \eqref{Ineq:LimHLS} can be written as
\be{Ineq:LimHLS2}
\ix{|\nabla(g\,\cinf)|^2}\le\ix{|f|^2\,\ninf}\;.
\ee
It is indeed well known that $\nabla(g\,\cinf)$ is in $L^2(\R^2,dx)$ as a solution of $-\Delta(g\,\cinf)=f\,\ninf$ only if \eqref{Cdt:OrthogConst}~holds. Lemma~\ref{Lem:LimHLS} will be improved in Section~\ref{Sec:Strict} (see Corollary~\ref{Cor:LimHLS}); the proof of such a result is independent of the remainder of this section.

\begin{proof}[Proof of Lemma~\ref{Lem:LimHLS}] To prove \eqref{Ineq:LimHLS}, we recall that the free energy $n\mapsto F[n]$ achieves its minimum for $n=\ninf$ according to the logarithmic Hardy-Littlewood-Sobolev inequality (see \cite{1101} for detailed considerations on this formulation of the inequality), and observe that $\mathsf Q_1[f]\ge 0$ for any smooth function $f$ with compact support satisfying \eqref{Cdt:Average}. The inequality then holds for any $f\in L^2(\R^2,\ninf\,dx)$ by density of smooth functions with compact support in $L^2(\R^2,\ninf\,dx)$.

If equality holds in \eqref{Ineq:LimHLS}, then the Euler-Lagrange equation amounts to $-\Delta f=f\,\ninf$, which characterizes the kernel $\mathop{Ker}(\L)$ according to Proposition~\ref{Prop:KerL}.\end{proof}

By \eqref{Ineq:LimHLS}, $\mathsf Q_1[f]$ is nonnegative, and positive semi-definite on the orthogonal of the kernel of $\L$, for the natural scalar product on $L^2(\ninf\,dx)$, \emph{i.e.}~for any $f\in L^2(\ninf\,dx)$ such that~\eqref{Cdt:Average} holds. Using previous notations, we may write
\[
\mathsf Q_1[f]=\ix{f\,(f-g\,\cinf)\,\ninf}\quad\mbox{with}\quad g\,\cinf=G_2*(f\,\ninf)\;.
\]
If \eqref{Cdt:OrthogConst} holds, we can also observe that
\[
\mathsf Q_1[f]=\ix{|f|^2\,\ninf}-\ix{|\nabla(g\,\cinf)|^2}\;.
\]

To $\mathsf Q_1$ we associate its polar form $\mathsf L_1$ defined on smooth functions with compact support such that \eqref{Cdt:Average} holds and define its Friedrich's extension on $L^2(\ninf\,dx)$, that we still denote by $\mathsf L_1$. By construction, $\mathsf L_1$ is a positive self-adjoint operator with domain $\mathcal D(\mathsf L_1)\subset L^2(\ninf\,dx)$. On $\mathcal D(\mathsf L_1)$, we shall denote by $\langle\cdot,\cdot\rangle$ the scalar product induced by $\mathsf L_1$. Explicitly, this means that
\[
\langle f,\tilde f\rangle=\ix{f\,\tilde f\,\ninf}+\frac 1{2\pi}\iint_{\R^2\times\R^2}f(x)\,\ninf(x)\,\log|x-y|\,\tilde f(y)\,\ninf(y)\;dx\,dy\;.
\]
The scalar product $\langle\cdot,\cdot\rangle$ induced by $\mathsf L_1$ is defined on the orthogonal of $f_{0,0}$, but can be extended as a bilinear form to $L^2(\ninf\,dx)$. If $f\in L^2(\ninf\,dx)$ is such that \eqref{Cdt:OrthogConst} holds, then we notice that
\[
\langle f,f_{0,0}\rangle=\ix{f\,(f_{0,0}-G_2*f_{0,0})\,\ninf}=0
\]
because $f_{0,0}=G_2*(f_{0,0}\,\ninf)+\mu_0'$. With these notations, notice that we have
\[
\langle f,f\rangle=\mathsf Q_1[f]\ge0
\]
for any $f\in\mathcal D(\mathsf L_1)$, with equality if and only if $f=0$.

\medskip We can also define the quadratic form $\mathsf Q_2$ as
\[
\mathsf Q_2[f]:=\ix{|\nabla(f-g\,\cinf)|^2\,\ninf}\quad\mbox{with}\quad g=\frac 1{\cinf}\,G_2*(f\,\ninf)\;.
\]
As for $\mathsf Q_1$, we define $\mathsf Q_2$ on the set of smooth functions such that \eqref{Cdt:Average} holds and extend it. The associated self-adjoint nonnegative operator is denoted by $\mathsf L_2$ and it is again a self-adjoint operator, with domain $\mathcal D(\mathsf L_2)\subset L^2(\ninf\,dx)$.
\begin{Prop}\label{Prop:LisSelfAdjoint} With the above notations, the restriction of $\L$ to $\mathcal D(\mathsf L_1)$ is a self-adjoint operator for the scalar product $\langle\cdot,\cdot\rangle$ with domain $\mathcal D(\mathsf L_2)$, such that
\[
\langle f,\L\,f\rangle=-\,\mathsf Q_2[f]\quad\forall\;f\in\mathcal D(\mathsf L_2)
\]
and $\mathrm{Ker}(\L)\cap\mathcal D(\mathsf L_2)=\{0\}$.
\end{Prop}

\begin{Rem} The function $f_{0,0}$ is an eigenfunction of $\L$ but this is not the case of $f\equiv1$. With the notations of Section~\ref{Sec:SomeEigenfunctions}, the functions $f_{0,1}$ and $f_{1,i}$ are orthogonal to $f\equiv1$ in $L^2(\R^2,\ninf\,dx)$ for $i=1$,~$2$, but this is not the case neither for $f_{0,0}$.
\end{Rem}

\subsection{The spectrum of {\texorpdfstring{$\mathcal L$}{L}} is discrete}\label{Sec:Discrete}

We define
\[
\Lambda_1:=\inf_{f\in\mathcal D(\mathsf L_2)\setminus\{0\}}\frac{\mathsf Q_2[f]}{\mathsf Q_1[f]}\quad\mbox{and}\quad\Lambda_\infty:=\lim_{R\to\infty}\inf_{\begin{array}{c}f\in\mathcal D(\mathsf L_2)\setminus\{0\}\\[2pt] \mathrm{supp}(f)\subset\R^2\setminus B(0,R) \end{array}}\frac{\mathsf Q_2[f]}{\mathsf Q_1[f]}\;.
\]
First, let us give a heuristic approach of the problem. As an application of Persson's method (see~\cite{MR0133586}), the bottom of the essential spectrum of $\L$ can be characterized as
\[
\inf\sigma_\mathrm{ess}(\L)=\Lambda_\infty\;.
\]
To prove that $\mathcal L$ has a spectral gap on $\mathcal D(\mathsf L_2)$, it is enough to show that $\Lambda_\infty$ is positive: either $\Lambda_1=\Lambda_\infty$, or $\Lambda_1<\Lambda_\infty$ is a nonnegative eigenvalue, which cannot be equal to $0$. This is summarized in the following statement.
\begin{Prop}\label{Prop:Ineq} With the above notations, $\Lambda_1$ is positive and 
\be{Gap}
\Lambda_1\,\mathsf Q_1[f]\le\mathsf Q_2[f]\quad\forall\;f\in\mathcal D(\mathsf L_2)\;.
\ee
\end{Prop}
For any $f\in\mathcal D(\mathsf L_2)$, if \eqref{Cdt:OrthogConst} holds, then Inequality~\eqref{Gap} can be reformulated as
\[
\Lambda_1\ix{|f|^2\,\ninf}\le\ix{|\nabla(f-g\,\cinf)|^2\,\ninf}+\Lambda_1\ix{|\nabla(g\,\cinf)|^2}\;.
\]

\medskip The proof of Proposition~\ref{Prop:Ineq} can be done by considering $\mathcal L$ as a perturbation of the operator $f\mapsto\ninf^{-1}\,\nabla\cdot(\ninf\,\nabla f)$ defined on $L^2(\R^2,\ninf\,dx)$. This was the method of \cite{Blanchet2010533}. However on such a space $\mathcal L$ is not self-adjoint and justifications are delicate because of the logarithmic kernel, away from the small mass regime.

In practice, Persson's method is not well designed either to handle convolution operators, although it can probably be adapted with little effort. This may even have been done, but we are not aware of such a result. Moreover, as we shall see below, we have: $\Lambda_\infty=\infty$, which further simplifies the proof. For these reasons, we will therefore give a direct proof, based on some of the tools of the concentration-compactness method (see \cite{MR653747,MR705681,MR778974,MR834360,MR850686}) and adapted to the case of a bounded measure, $\ninf\,dx$, as in~\cite{blanchet:46}. In that framework, $\Lambda_\infty$ corresponds to the \emph{problem at infinity.} For simplicity, let us split the proof into Lemmas~\ref{Lem:EssSpectrum}~and~\ref{Lem:DiscreteSpectrum}.
\begin{Lemma}\label{Lem:EssSpectrum} With the above notations, $\Lambda_\infty=\infty$.\end{Lemma}
\begin{proof} Recall that
\[
\ninf=M\,\frac{e^{\cinf-\frac 12\,|x|^2}}{\ix{e^{\cinf-\frac 12\,|x|^2}}}
\]
where $\cinf=(-\Delta)^{-1}\,\ninf$ is such that
\[
\limsup_{|x|\to\infty}\big|\,\cinf(x)+\frac M{2\pi}\,\log|x|\,\big|<\infty\;.
\]
As a consequence, we know that
\[
\ninf(x)\sim |x|^{-\alpha}\,e^{-\frac 12\,|x|^2}\quad\mbox{as}\quad|x|\to+\infty\;,\quad\mbox{with}\quad\alpha=\frac M{2\pi}\;.
\]
We can expand the square $|\nabla(f-g\,\cinf)|^2$ and get
\begin{multline*}
\mathsf Q_2[f]=\ix{|\nabla(f-g\,\cinf)|^2\,\ninf}\\
=\ix{|\nabla f|^2\,\ninf}+\ix{|\nabla(g\,\cinf)|^2\,\ninf}\hspace*{2cm}\\
+2\ix{f\,\nabla(g\,\cinf)\cdot\nabla\ninf}-\,2\ix{f\,\big(-\Delta(g\,\cinf)\big)\,\ninf}\;.
\end{multline*}
Assume that $f$ is supported in $\R^2\setminus B(0,R)$, for $R>0$, large. Then 
\begin{multline*}
\ix{f\,\big(-\Delta(g\,\cinf)\big)\,\ninf}=\ix{|f|^2\,\ninf^2}\\
\le\sup_{|x|>R}\ninf(x)\ix{|f|^2\,\ninf}\sim R^{-\alpha}\,e^{-\frac 12\,R^2}\ix{|f|^2\,\ninf}
\end{multline*}
on the one hand, and we know from Persson's method that
\[
\lim_{R\to\infty}\inf_{\begin{array}{c}f\in\mathcal D(\mathsf L_2)\setminus\{0\}\\
\mathrm{supp}(f)\subset\R^2\setminus B(0,R)
\end{array}}\frac{\ix{|\nabla f|^2\,\ninf}}{\ix{|f|^2\,\ninf}}=+\infty
\]
on the other hand, so that, for any $\varepsilon>0$, there exists $R>0$ large enough for which
\[
\ix{|f|^2\,\ninf}\le\varepsilon\ix{|\nabla f|^2\,\ninf}
\]
for any function $f\in H^1(\R^2,\ninf\,dx)$. Equivalently, we can write that there exists a positive function $R\mapsto\varepsilon(R)$ such that $\lim_{R\to+\infty}\varepsilon(R)=0$ and
\[
0\le\ix{f\,\big(-\Delta(g\,\cinf)\big)\,\ninf}\le R^{-\alpha}\,e^{-\frac 12\,R^2}\,\varepsilon(R)
\]
for any function $f\in H^1$ such that $\mathrm{supp}(f)\subset\R^2\setminus B(0,R)$. 

\medskip Assume first that Condition \eqref{Cdt:OrthogConst} is satisfied. We notice that
\[
\left|\ix{f\,\nabla(g\,\cinf)\cdot\nabla\ninf}\right|\le2\ix{\left|f\,\sqrt{\ninf}\right|\;\left|\nabla(g\,\cinf)\right|\;\left|(\nabla \cinf-x)\,\sqrt{\ninf}\right|}
\]
can be estimated by
\[
\left|\ix{f\,\nabla(g\,\cinf)\cdot\nabla\ninf}\right|\le2\sup_{|x|>R}\left|(\nabla \cinf-x)\,\sqrt\ninf\right|\(\ix{|f|^2\,\ninf}\ix{|\nabla(g\,\cinf)|^2}\)^{1/2}.
\]
As a consequence of~\eqref{Ineq:LimHLS2}, we find that
\begin{multline*}
\left|\ix{f\,\nabla(g\,\cinf)\cdot\nabla\ninf}\right|\le2\sup_{|x|>R}\left|(\nabla \cinf-x)\,\sqrt\ninf\right|\ix{|f|^2\,\ninf}\\
\le2\sup_{|x|>R}\left|(\nabla \cinf-x)\,\sqrt\ninf\right|\,\varepsilon(R)\ix{|\nabla f|^2\,\ninf}\;.
\end{multline*}
On the other hand, since Condition \eqref{Cdt:OrthogConst} is satisfied, we know for free that
\[
\mathsf Q_1[f]=\ix{|f|^2\,\ninf}-\ix{|\nabla(g\,\cinf)|^2}\le\ix{|f|^2\,\ninf}\;.
\]
As a consequence, we have obtained that
\[
\lim_{R\to\infty}\kern -4pt\inf_{\begin{array}{c}f\in\mathcal D(\mathsf L_2)\setminus\{0\}\\
\mathrm{supp}(f)\subset\R^2\setminus B(0,R)\end{array}}\kern -4pt\frac{\mathsf Q_2[f]}{\mathsf Q_1[f]}
=\lim_{R\to\infty}\kern -4pt\inf_{\begin{array}{c}f\in\mathcal D(\mathsf L_2)\setminus\{0\}\\
\mathrm{supp}(f)\subset\R^2\setminus B(0,R)\end{array}}\kern -4pt\frac{\ix{|\nabla f|^2\,\ninf}}{\ix{|f|^2\,\ninf}}=+\infty\;,
\]
which proves our claim.

\medskip If Condition \eqref{Cdt:OrthogConst} is not satisfied, the proof is more complicated. By homogeneity, there is no restriction to assume that $\frac 1M\ix{f^2\,\ninf}=1$. Let $\theta:=\frac 1M\ix{f\,\ninf}$ and $\tilde f:=f-\theta$, $\tilde g:=g-\theta$. Then $\ix{\tilde f\,\ninf}=0$. Notice that, by the Cauchy-Schwarz inequality, we have: $\theta\in[-1,1]$. Moreover, if $\theta\neq0$, then $B(0,R)$ is contained in $\mathrm{supp}(\tilde f)$.

With these notations, we first have to estimate
\[
\ix{f\,\nabla(g\,\cinf)\cdot\nabla\ninf}=2\,\theta\ix{f\,\sqrt\ninf\,\nabla\cinf\cdot\nabla\sqrt\ninf}+2\ix{f\,\sqrt\ninf\,\nabla(\tilde g\,\cinf)\cdot\nabla\sqrt\ninf}\;.
\]
By the Cauchy-Schwarz inequality, we have
\[
\left|\ix{f\,\nabla\cinf\cdot\nabla\sqrt\ninf}\right|^2\le\ix{f^2\,\ninf}\int_{\R^2\setminus B(0,R)}|\nabla\cinf|^2\,|\nabla\sqrt\ninf|^2\,dx\;,
\]
and it is simple to check that the last integral in the right hand side converges to $0$ as $R\to\infty$. The second integral can be estimated as before by writing
\[
\left|\ix{f\,\sqrt\ninf\,\nabla(\tilde g\,\cinf)\cdot\nabla\sqrt\ninf}\right|^2\le\sup_{|x|>R}\left|\nabla\sqrt\ninf\right|^2\ix{|f|^2\,\ninf}\ix{|\nabla(\tilde g\,\cinf)|^2}
\]
and by recalling that $\ix{|\tilde f|^2\,\ninf}=M$. From these estimates, we conclude that
\[
\lim_{R\to\infty}\kern -4pt\inf_{\begin{array}{c}f\in\mathcal D(\mathsf L_2)\setminus\{0\}\\
\mathrm{supp}(f)\subset\R^2\setminus B(0,R)\end{array}}\kern -4pt\mathsf Q_2[f]=\infty\;.
\]
We also need to estimate $\mathsf Q_1[f]$ and this can be done by showing that
\[
\mathsf Q_1[f]=M+\frac 1{2\pi}\iint_{\R^2\times\R^2}(f\,\ninf)(x)\,\log|x-y|\,(f\,\ninf)(y)\;dx\,dy
\]
is bounded from above if we still impose that $\ix{|f|^2\,\ninf}=M$. Using the crude estimate
\[
2\,\log|x-y|\le|x-y|^2\le2\,(|x|^2+|y|^2)\quad\forall\;(x,y)\in\R^2\times\R^2
\]
and, as a consequence, 
\[
\iint_{\R^2\times\R^2}(f\,\ninf)(x)\,\log|x-y|\,(f\,\ninf)(y)\;dx\,dy\le2\ix{f\,\ninf}\int_{\R^2}|y|^2\,(f\,\ninf)(y)\;dy\;,
\]
we conclude by observing that
\begin{multline*}
\ix{f\,\ninf}\le\sqrt M\,\(\int_{\R^2\setminus B(0,R)}\ninf\;dx\)^{1/2}\\
\mbox{and}\;\int_{\R^2}|y|^2\,(f\,\ninf)(y)\;dy\le\sqrt M\,\(\int_{\R^2\setminus B(0,R)}|y|^4\,\ninf\;dy\)^{1/2}
\end{multline*}
both converge to $0$ as $R\to\infty$.\end{proof}

\begin{Lemma}\label{Lem:DiscreteSpectrum} With the above notations, $\Lambda_1>0$.\end{Lemma}
\begin{proof} Tools for the proof of this lemma are to a large extent standard in concentration-compactness methods or when applied to models of quantum chemistry like in \cite{CL1}, so we shall only sketch the main steps and omit as much as possible the technicalities of such an approach. We will actually prove a result that is stronger than the one of Lemma~\ref{Lem:DiscreteSpectrum}: $\Lambda_1$ is achieved by some function $f\in\mathcal D(\mathsf L_2)$.

Consider a minimizing sequence $(f_n)_{n\in\N}$ for the functional $f\mapsto\mathsf Q_2[f]/\mathsf Q_1[f]$ defined on $\mathcal D(\mathsf L_2)\setminus\{0\}$. By homogeneity, we may assume that $\ix{f^2\,\ninf}=1$ for any $n\in\N$, with no restriction, while $\mathsf Q_2[f_n]$ is bounded uniformly in $n\in\N$. Let $F_n:=f_n\,\sqrt\ninf$. In the framework of concentration-compactness methods, for any given $\varepsilon>0$, it is a standard result that one can decompose $F_n$ as
\[
F_n=F_n^{(1)}+F_n^{(2)}+\widetilde F_n
\]
for any $n\in\N$, with
\[
\ix{|F_n^{(1)}|^2}+\ix{|F_n^{(2)}|^2}+\ix{|\widetilde F_n|^2}=1
\]
where $F_n^{(1)}$, $F_n^{(2)}$ and $\widetilde F_n$ are supported respectively in $B(0,2\,R)$, $\R^2\setminus B(0,R_n)$ and $B(0,2\,R_n)\setminus B(0,R)$, $R_n>R>1$, $\lim_{n\to\infty}R_n=\infty$,
\[
\ix{|F_n^{(1)}|^2}\ge\theta-\varepsilon\quad\mbox{and}\quad\ix{|F_n^{(2)}|^2}\ge 1-\theta-\varepsilon
\]
for some $\theta\in[0,1]$, which is independent of $\varepsilon$. As a consequence, we also have that $\ix{|\widetilde F_n|^2}\le 2\,\varepsilon$. A standard method to obtain such a decomposition is based on the IMS truncation method, which goes as follows. Take a smooth truncation function $\chi$ with the following properties: $0\le\chi\le1$, $\chi(x)=1$ for any $x\in B(0,1)$, $\chi(x)=0$ for any $x\in\R^2\setminus B(0,2)$, and define $\chi_R(x):=\chi(x/R)$ for any $x\in\R^2$. Then for an appropriate choice of $R$ and $(R_n)_{n\in\N}$, we can choose
\[
F_n^{(1)}=\chi_R\,F_n\quad\mbox{and}\quad F_n^{(2)}=\sqrt{1-\chi_{R_n}^2}\,F_n\;.
\]
If $\theta=1$, then $(F_n^{(1)})_{n\in\N}$ strongly converges in $L^2_{\rm loc}(\R^2,dx)$ to some limit $F$ and we have
\[
\liminf_{n\to\infty}\ix{|\nabla F_n|^2}\ge\ix{|\nabla F|^2}\quad\mbox{and}\quad\ix{|F|^2}\ge1-\varepsilon\;.
\]
Now we repeat the argument as $\varepsilon=\varepsilon_n\to 0_+$, take a diagonal subsequence that we still denote by $(F_n)_{n\in\N}$, define 
\[
F_n^{(1)}=\chi_{R_n^{(1)}}\,F_n\;,\quad F_n^{(2)}=\sqrt{1-\chi_{R_n^{(2)}}^2}\,F_n\quad\mbox{and}\quad \widetilde F_n=F_n-F_n^{(1)}-F_n^{(2)}\;,
\]
where $\lim_{n\to\infty}R_n^{(1)}=+\infty$ and $R_n^{(2)}\ge 2\,R_n^{(1)}$ for any $n\in\N$. Since limits obtained above by taking a diagonal subsequence coincide on larger and larger centered balls (that is when $R$ increases), we find a nontrivial minimizer $f=F/\sqrt\ninf$ such that $\ix{F^2}=\ix{f^2\,\ninf}=1$, since all other terms are relatively compact. Notice that $\mathsf Q_1[f]>0$ because the condition~\eqref{Cdt:Average} is preserved by passing to the limit. Hence $\Lambda_1$ is achieved and we know that $\Lambda_1$ is positive because~$\mathsf Q_1$ is positive semi-definite. 

Assume now that $\theta<1$. We know that $\ix{|\widetilde F_n|^2}\le 2\,\varepsilon_n$ and hence
\[
\lim_{n\to\infty}\ix{|\widetilde F_n|^2}=0\;.
\]
It is not difficult to see that cross terms do not play any role in the integrals involving convolution kernels, as it is standard for Hartree type (or Schr\"odinger-Poisson) models. As a consequence, we can write that
\[
\lim_{n\to\infty}\mathsf Q_1[f_n]=\lim_{n\to\infty}\mathsf Q_1[f_n^{(1)}]+\lim_{n\to\infty}\mathsf Q_1[f_n^{(2)}]
\]
where $f_n^{(i)}:=F_n^{(i)}/\sqrt\ninf$, $i=1$, $2$. Proceeding as above, we may find a limit $f$ of $(f_n^{(1)})_{n\in\N}$, in $L^2(\ninf\,dx)$. It is then straightforward to observe that
\[
\Lambda_1=\lim_{n\to\infty}\frac{\mathsf Q_2[f_n]}{\mathsf Q_1[f_n]}\ge\lim_{n\to\infty}\frac{\mathsf Q_2[f]+\mathsf Q_2[f_n^{(2)}]}{\mathsf Q_1[f]+\mathsf Q_1[f_n^{(2)}]}\;.
\]
If $\theta>0$, we know that $\mathsf Q_2[f]\ge\Lambda_1\,\mathsf Q_1[f]$ and $\lim_{n\to\infty}\mathsf Q_2[f_n^{(2)}]/\mathsf Q_1[f_n^{(2)}]>\Lambda_1$ by Lemma~\ref{Lem:EssSpectrum}, so that $\lim_{n\to\infty}\mathsf Q_2[f_n^{(2)}]=0$ and $f$ is a nontrivial minimizer: we are back to the case $\theta=1$, but with a different normalization of $f$. If $\theta=0$, it is clear that $\mathsf Q_1[f]=0$ and we get
\[
\Lambda_1\ge\lim_{n\to\infty}\frac{\mathsf Q_2[f_n^{(2)}]}{\mathsf Q_1[f_n^{(2)}]}=\infty\;,
\]
again by Lemma~\ref{Lem:EssSpectrum}, a contradiction with the fact that $\mathsf Q_2[f]/\mathsf Q_1[f]$ takes finite values for arbitrary test functions in $\mathcal D(\mathsf L_2)\setminus\{0\}$. This concludes our proof.
\end{proof}

\begin{Rem} For any $k\in\N$, $k\ge1$, define the Raleigh quotient
\[
\Lambda_k:=\inf_{\begin{array}{c}f\in\mathcal D(\mathsf L_2)\setminus\{0\}\\\langle f_j,f\rangle=0\,,\;j=0,\,1,...k-1\,\end{array}}\frac{\mathsf Q_2[f]}{\mathsf Q_1[f]}
\]
where $f_j$ denotes a critical point associated to $\Lambda_j$. Critical points are counted with multiplicity. Since the orthogonality condition $\langle f_j,f\rangle=0$ is preserved by taking the limit along the weak topology of $L^2(\ninf\,dx)$, building a minimizing sequence for $k\ge1$ goes as in the case $k=1$. It is easy to check that $\Lambda_k$ is then an eigenvalue of $\L$ considered as an operator on $\mathcal D(\mathsf L_2)$ with scalar product $\langle\cdot,\cdot\rangle$, for any $k\ge 1$and $\lim_{k\to\infty}\Lambda_k=\infty$.
\end{Rem}

\subsection{A spectral gap inequality}\label{Sec:SpectralGap}

We are now going to prove that Ineq.~\eqref{Gap} holds with $\Lambda_1=\lambda_{1,i}=1$, $i=1$, $2$. This is our first main estimate.
\begin{Teo}\label{Teo:Gap} For any function $f\in\mathcal D(\mathsf L_2)$, we have
\[
\mathsf Q_1[f]\le\mathsf Q_2[f]\;.
\]
\end{Teo}
Recall that, with the notations of Section~\ref{Sec:Linearization}, $\mathsf Q_1[f]=\langle f,f\rangle$ and $\mathsf Q_2[f]=\langle f,\,\L\,f\rangle$.

\begin{proof} We have to compute the lowest positive eigenvalue of $\L$. After a reformulation in terms of cumulated densities for the solution of \eqref{RKS} and for the eigenvalue problem for $\mathcal L$, we will identify the lowest eigenvalue $\lambda_{0,1}=2$ when $\mathcal L$ is restricted to radial functions, and the lowest ones, $\lambda_{1,1}=\lambda_{1,2}=1$, when $\mathcal L$ is restricted to functions corresponding to the $k=1$ component in the decomposition into spherical harmonics.

\medskip\noindent\emph{Step 1.~Reformulation in terms of cumulated densities.} 

Among spherically symmetric functions, it is possible to reduce the problem to a single ordinary differential equation.

Consider first a stationary solution $(\ninf,\cinf)$ of~\eqref{RKS} and as in \cite{0039} or \cite{springerlink:10.1007/s00285-010-0357-5} (also see references therein), let us rewrite the system in terms of the \emph{cumulated densities} $\Phi$ and $\Psi$ defined by
\begin{eqnarray*}
&&\Phi(s):=\frac 1{2\pi}\int_{B(0,\sqrt s)}\ninf(x)\;dx\;,\\
&&\Psi(s):=\frac 1{2\pi}\int_{B(0,\sqrt s)}\cinf(x)\;dx\;.
\end{eqnarray*}
Notice that $\Phi(s)=\frac 1{2\,\pi}\,M_\eps(\pi\,s)$ for $\eps=0$, with the notations of the proof of Theorem~\ref{Thm:UniformBound}. The motivation for such a reformulation is that the system can be rewritten in terms of a nonlinear, local, ordinary differential equation for $\Phi$ using the fact that $\ninf$ is radial. With a slight abuse of notations, we can consider $\ninf$ and $\cinf$ as functions of $r=|x|$. Elementary computations show that
\[
\ninf(\sqrt s)=2\,\Phi'(s)\quad\mbox{and}\quad\ninf'(\sqrt s)=4\,\sqrt s\,\Phi''(s)\;,
\]
\[
\cinf(\sqrt s)=2\,\Psi'(s)\quad\mbox{and}\quad\cinf'(\sqrt s)=4\,\sqrt s\,\Psi''(s)\;.
\]
After one integration with respect to $r=\sqrt s$, the Poisson equation $-\Delta\cinf=\ninf$ can be rewritten as
\[
-\sqrt s\,\cinf'(\sqrt s)=\Phi(s)
\]
while the equation for $\ninf$, after an integration on $(0,r)$, is
\[
\ninf'(\sqrt s)+\sqrt s\,\ninf(\sqrt s)-\ninf(\sqrt s)\,\cinf'(\sqrt s)=0\;.
\]
These two equations written in terms of $\Phi$ and $\Psi$ are
\[
-\,4\,s\,\Psi''=\Phi
\]
and
\[
\Phi''+\frac 12\,\Phi'-\,2\,\Phi'\,\Psi''\,.
\]
After eliminating $\Psi''$, we find that $\Phi$ is the solution of the ordinary differential equation
\be{Eqn:CumulatedDensities}
\Phi''+\frac 12\,\Phi'+\frac 1{2\,s}\,\Phi\,\Phi'=0
\ee
with initial conditions $\Phi(0)=0$ and $\Phi'(0)=\frac 12\,n(0)=:a$, so that all solutions can be para\-metrized in terms of $a>0$.

\medskip Consider next the functions $f$ and $g$ involved in the linearized Keller-Segel system~\eqref{RKS} and define the corresponding \emph{cumulated densities} given by
\begin{eqnarray*}
&&\phi(s):=\frac 1{2\pi}\int_{B(0,\sqrt s)}(f\,\ninf)(x)\;dx\;,\\
&&\psi(s):=\frac 1{2\pi}\int_{B(0,\sqrt s)}(g\,\cinf)(x)\;dx\;.
\end{eqnarray*}
If $g\,\cinf=(-\Delta)^{-1}(f\,\ninf)$ and $f$ is a solution of the eigenvalue problem
\[
-\,\L\,f=\lambda\,f\;,
\]
then we can make a computation similar to the above one and get
\[
(\ninf\,f)(\sqrt s)=2\,\phi'(s)\;,\quad(\ninf\,f')(\sqrt s)=4\,\sqrt s\,\phi''(s)-\,2\,\frac{\ninf'}{\ninf}(\sqrt s)\,\phi'(s)\;,
\]
\[
(g\,\cinf)(\sqrt s)=2\,\psi'(s)\quad\mbox{and}\quad(g\,\cinf)'(\sqrt s)=4\,\sqrt s\,\psi''(s)\;.
\]
The equations satisfied by $f$ and $g$ are
\[
-\sqrt s\,(g\,\cinf)'(\sqrt s)=\phi(s)
\]
and 
\[
\sqrt s\,\Big((\ninf\,f')(\sqrt s)-\ninf\,(g\,\cinf)'(\sqrt s)\Big)+\lambda\,\phi(s)=0\;.
\]
These two equations written in terms of $\phi$ and $\psi$ become
\[
-\,4\,s\,\psi''=\phi
\]
and
\[
4\,s\(\phi''-\frac{\Phi''}{\Phi'}\,\phi'-\,2\,\Phi'\,\psi''\)+\lambda\,\phi=0\;.
\]
After eliminating $\psi''$, we find that $\phi$ is the solution of the ordinary differential equation
\[
\phi''-\frac{\Phi''}{\Phi'}\,\phi'+\frac{\lambda+2\,\Phi'}{4\,s}\,\phi=0\;.
\]
Taking into account the equation for $\Phi$, that is
\[
-\frac{\Phi''}{\Phi'}=\frac 12+\frac\Phi{2\,s}\;,
\]
we can also write that $\phi$ solves
\be{Eqn:CumulatedDensitiesLinearized}
\phi''+\frac{s+\Phi}{2\,s}\,\phi'+\frac{\lambda+2\,\Phi'}{4\,s}\,\phi=0\;.
\ee

Recall that the set of solutions to \eqref{Eqn:CumulatedDensities} is parametrized by $a=\Phi'(0)$. It is straightforward to remark that $\phi=\frac d{da}\Phi$ solves \eqref{Eqn:CumulatedDensitiesLinearized} with $\lambda=0$. The reader is invited to check that $s\mapsto s\,\Phi'(s)$ provides a nonnegative solution of \eqref{Eqn:CumulatedDensitiesLinearized} with $\lambda=2$.

\medskip\noindent\emph{Step 2.~Characterization of the radial ground state.}

It is possible to rewrite~\eqref{Eqn:CumulatedDensitiesLinearized} as
\[
\frac d{ds}\,\(e^{\alpha(s)}\,\frac{d\phi}{ds}\)+\frac{\lambda+2\,\Phi'}{4\,s}\,e^{\alpha(s)}\,\phi=0\quad\mbox{with}\quad\alpha(s):=\frac s2+\frac 12\int_0^s\frac{\phi(\sigma)}\sigma\;d\sigma\;.
\]
The equation holds on $(0,\infty)$ and boundary conditions are $\phi(0)=0$ and $\lim_{s\to\infty}\phi(s)=0$. By the Sturm-Liouville theory, we know that $\lambda=2=\lambda_{0,1}$ is then the lowest positive eigenvalue such that~$\phi$ is nonnegative and satisfies the above boundary conditions. 

In other words, we have shown that the function $f_{0,1}$ found in Section~\ref{Sec:SomeEigenfunctions} generates the eigenspace corresponding to the lowest positive eigenvalue of $\L$ restricted to radial functions.

\medskip\noindent\emph{Step 3.~Spherical harmonics decomposition.}

We have to deal with non-radial modes of $\L$. Since $\ninf$ and $\cinf$ are both radial, we can use a spherical harmonics decomposition for that purpose. As in \cite{KS-num2011}, the eigenvalue problem for the operator $\mathcal L$ amounts to solve among radial functions $f$ and $g$ the system
\begin{eqnarray*}
&&-f''-\frac 1r\,f'+\frac{k^2}{r^2}\,f+(r-\cinf')\,(f'-(g\,\cinf)')-\ninf\,f=\lambda\,f\;,\\
&&-(g\,\cinf)''-\frac 1r\,(g\,\cinf)'+\frac{k^2}{r^2}\,(g\,\cinf)=\ninf\,f\;,
\end{eqnarray*}
for some $k\in\N$, $k\ge 1$. Here as above, we make the standard abuse of notations that amounts to write $\ninf$ and $\cinf$ as a function of $r=|x|$. It is straightforward to see that $k=1$ realizes the infimum of the spectrum of $\L$ among non-radial functions. The function $f=-\ninf'$ provides a nonnegative solution for $k=1$ and $\lambda=1$. It is then possible to conclude using the following observation: $f$ is a radial $C^2$ solution if and only if $r\mapsto r\,f=:\tilde f(r)$ solves $-\,\L\,\tilde f=(\lambda+1)\,\tilde f$ among radial functions, and we are back to the problem studied in Step 2. The value we look for is therefore $\lambda=1=\lambda_{1,1}=\lambda_{1,2}$.
\begin{figure}[hb]
\includegraphics[width=7cm]{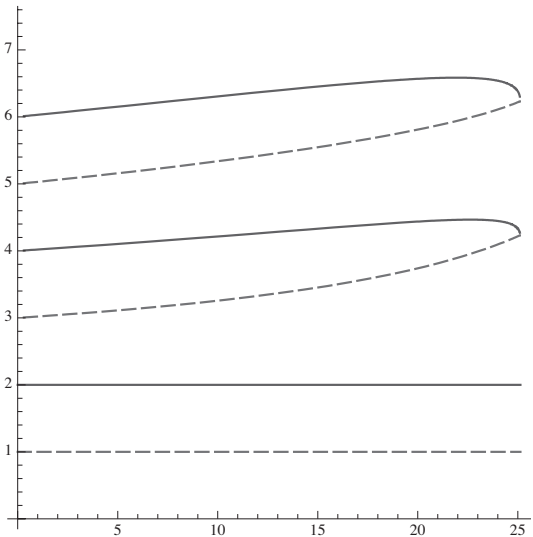}
\caption{\sl Using a shooting method, one can numerically compute the lowest eigenvalues of $-\mathcal L$ for \hbox{$k=0$} (radial functions, plain curves) and for the $k=1$ component of the spherical harmonics decomposition (dashed curves). The plot shows that $1$ and $2$ are the lowest eigenvalues, when mass varies between $0$ and $8\,\pi\approx25.1327$. See \cite{KS-num2011} for details and further numerical results.}
\end{figure}

In other words, we have shown that the functions $f_{1,1}$ and $f_{1,2}$ found in Section~\ref{Sec:SomeEigenfunctions} generate the eigenspace corresponding to the lowest positive eigenvalue of $\L$ corresponding to $k=1$. We are now in position to conclude the proof of Theorem~\ref{Teo:Gap}.

Either the spectral gap is achieved among radial functions and $\Lambda_1=2$, or it is achieved among functions in one of the non-radial components corresponding to the spherical harmonics decomposition: the one given by $k=1$ minimizes the gap and hence we obtain $\Lambda_1=1$. See Fig.~1 for an illustration.\end{proof}

As a consequence of Step 2 in the proof of Theorem~\ref{Teo:Gap}, we find that the following inequality holds.
\begin{Prop}\label{Prop:RadialGap} For any radial function $f\in\mathcal D(\mathsf L_2)$, we have
\[
2\,\mathsf Q_1[f]\le\mathsf Q_2[f]\;.
\]
where, with the notations of Section~\ref{Sec:Linearization}, $\mathsf Q_1[f]=\langle f,f\rangle$ and $\mathsf Q_2[f]=\langle f,\,\L\,f\rangle$.\end{Prop}
This observation has to be related with recent results of V.~Calvez and J.A.~Carrillo. As a consequence, the rate $e^{-\,2\,t}$ in Theorem~\ref{Thm:Main} can be replaced by $e^{-\,4\,t}$ when solutions are radially symmetric, consistently with \cite[Theorem 1.2]{Calvez:2010fk}. The necessary adaptations (see Section~\ref{Sec:PrfThmMain}) are straightforward.

\section{A strict positivity result for the linearized entropy}\label{Sec:Strict}

Lemma~\ref{Lem:LimHLS} can be improved and this is our second main estimate.
\begin{Teo}\label{Teo:LimHLS} There exists $\Lambda>1$ such that
\be{Ineq:ImprovedGap}
\Lambda\ix{f\,\ninf\,(-\Delta)^{-1}\,(f\,\ninf)}\le\ix{|f|^2\,\ninf}
\ee
for any $f\in L^2(\R^2,\ninf\,dx)$ such that \eqref{Cdt:Average} holds.
\end{Teo}
\begin{proof} Let us give an elementary proof based on two main observations: the equivalence with a Poincar\'e type inequality using Legendre's transform, and the application of a concentration-compactness method for proving the Poincar\'e inequality. Recall that by Lemma~\ref{Lem:LimHLS} we already know that \eqref{Ineq:ImprovedGap} holds with $\Lambda=1$.

\medskip\noindent\emph{Step 1.} We claim that Inequality~\eqref{Ineq:ImprovedGap} is equivalent to 
\be{Ineq:ImprovedGapBis}
\Lambda\ix{|h|^2\,\ninf}\le\ix{|\nabla h|^2}
\ee
for any $h\in L^2(\R^2,\ninf\,dx)$ such that the condition $\ix{h\,f_{0,0}\,\ninf}=0$ holds, \emph{i.e.}~such that $h$ satisfies \eqref{Cdt:Average}. Let us prove this claim. 

Assume first that~\eqref{Ineq:ImprovedGapBis} holds and take Legendre's transform of both sides with respect to the natural scalar product in $L^2(\ninf\,dx)$: for any $f\in L^2(\R^2,\ninf\,dx)$ such that \eqref{Cdt:Average} holds,
\[
\sup_h\(\ix{f\,h\,\ninf}-\frac 12\ix{h^2\,\ninf}\)\ge\sup_h\(\ix{f\,h\,\ninf}-\frac1{2\Lambda}\ix{|\nabla h|^2}\)
\]
where the supremum is taken on both sides on all functions $h$ in $L^2(\R^2,\ninf\,dx)$ such that~$h$ satisfies \eqref{Cdt:Average}. Since semi-definite positive quadratic forms are involved, the suprema are achieved by convexity. For the left hand side, we find that the optimal function satisfies
\[
f=h+\mu\,f_{0,0}
\]
for some Lagrange multiplier $\mu\in\R$. However, if we multiply by $f_{0,0}\,\ninf$, we get that $\mu=0$, so that the left hand side of the inequality is simply $\frac 12\ix{f^2\,\ninf}$. As for the right hand side, we find that the optimal function $f$ is such that
\[
f\,\ninf=-\frac1\Lambda\,\Delta h+\mu\,f_{0,0}\,\ninf
\]
for some Lagrange multiplier $\mu\in\R$. In that case, if we multiply by $(-\Delta)^{-1}(f_{0,0}\,\ninf)=f_{0,0}$, we get that
\[
\mu\ix{f_{0,0}^2\,\ninf}=\frac1\Lambda\ix{\Delta h\,(-\Delta)^{-1}(f_{0,0}\,\ninf)}=-\frac1\Lambda\ix{h\,f_{0,0}\,\ninf}=0
\]
thus proving that $\mu=0$ as well. Hence the right hand side of the inequality is simply $\frac\Lambda2\ix{f\,\ninf\,(-\Delta)^{-1}\,(f\,\ninf)}$, which establishes \eqref{Ineq:ImprovedGap}. It is left to the reader to check that Inequality~\eqref{Ineq:ImprovedGapBis} can also be deduced from \eqref{Ineq:ImprovedGap} by a similar argument.

\medskip\noindent\emph{Step 2.} Let us prove that \eqref{Ineq:ImprovedGapBis} holds for some $\Lambda>1$. Consider an optimizing sequence of functions $(h_n)_{n\ge1}$ such that $\ix{h_n^2\,\ninf}=1$ and $\ix{h_n\,f_{0,0}\,\ninf}=0$ for any $n\ge 1$, and $\lim_{n\to\infty}\ix{|\nabla h_n|^2}=\Lambda$. As in the proof of Lemma~\ref{Lem:DiscreteSpectrum}, we are going to use the IMS truncation method. Consider a smooth function $\chi$ with the following properties: $0\le\chi\le1$, $\chi(x)=1$ for any $x\in B(0,1)$, $\chi(x)=0$ for any $x\in\R^2\setminus B(0,2)$, and define $\chi_R(x):=\chi(x/R)$ for any $x\in\R^2$. It is standard in concentration-compactness methods that for any $\varepsilon>0$, one can find a sequence of positive numbers $(R_n)_{n\ge1}$ such that
\[
h_n^{(1)}=\chi_{R_n}\,h_n\quad\mbox{and}\quad h_n^{(2)}=\sqrt{1-\chi_{R_n}^2}\,h_n\;,
\]
and, up to the extraction of a subsequence, there exists a function $h$ such that
\[
\ix{|\nabla h_n^{(1)}|^2}\ge\eta\,\Lambda-\varepsilon\quad\mbox{and}\quad\ix{|\nabla h_n^{(2)}|^2}\ge (1-\eta)\,\Lambda-\varepsilon
\]
for some $\eta\in[0,1]$, where the sequence $(\nabla h_n^{(1)})_{n\ge1}$ strongly converges to $\nabla h$ and
\[
\lim_{n\to\infty}\ix{|h_n^{(1)}|^2\,\ninf}=\ix{h^2\,\ninf}=:\theta
\]
(this implies the strong convergence of $(h_n^{(1)})_{n\ge1}$ towards $h$ in $L^2(\ninf\,dx)$) because
\[
\int_{\R^2\setminus B(0,R)}|h_n^{(1)}|^2\,\ninf\;dx\le\(\ix{|h_n^{(1)}|^\frac{2d}{d-2}\,\ninf}\)^\frac{d-2}d \(\int_{\R^2\setminus B(0,R)}\ninf^\frac d2\,dx\)^\frac 2d
\]
is uniformly small as $R\to\infty$ by Sobolev's inequality and because the last term of the right hand side is such that $\lim_{R\to\infty}\int_{\R^2\setminus B(0,R)}\ninf^{d/2}\,dx=0$. Of course, we know that
\[
\eta\,\Lambda\ge\ix{|\nabla h|^2}\ge\Lambda\,\theta
\]
by definition of $\Lambda$. The above estimate also guarantees that 
\[
\int_{\R^2\setminus B(0,R)}|h_n^{(2)}|^2\,\ninf\;dx=:\varepsilon_n\to 0\quad\mbox{as}\quad n\to\infty\;.
\]
By construction of $(h_n^{(1)})_{n\ge1}$ and $(h_n^{(2)})_{n\ge1}$, we know that $|h_n^{(1)}|^2+|h_n^{(2)}|^2=|h_n|^2$ and hence \hbox{$\theta=1$}. This also means that $\eta=1$ and hence $h$ is a minimizer, since the constraint passes to the limit:
\[
\ix{h\,f_{0,0}\,\ninf}=\lim_{n\to\infty}\ix{h_n\,f_{0,0}\,\ninf}=0\;.
\]
The function $h$ is a solution of the Euler-Lagrange equation:
\[
-\Delta h=\Lambda\,h\,\ninf\;.
\]
By Proposition~\ref{Prop:KerL}, if $\Lambda=1$, then $h$ and $f_{0,0}$ are collinear, which is a contradiction with the constraint. This proves that $\Lambda>1$.\end{proof}

Notice that the functions $\ninf$ and $\cinf$ being radial symmetric, we know that a decomposition into spherical harmonics allows to reduce the problem of computing all eigenvalues to radially symmetric eigenvalue problems. This provides a method to compute the explicit value of $\Lambda$, at least numerically.

\begin{Rem} Inequality~\eqref{Ineq:ImprovedGapBis} is a Poincar\'e inequality, which has already been established in~\cite{CD2012Cras} as a linearized version of an Onofri type inequality. This Onofri inequality is dual of the logarithmic Hardy-Littlewood-Sobolev type inequality that has been established in \cite{MR2226917} and according to which the free energy functional $F[n]$ is nonnegative. In this paper, we make use of the linearized versions of these Onofri and subcritical (in terms of the mass) logarithmic Hardy-Littlewood-Sobolev inequalities, and of improved forms that are obtained by requiring some addition orthogonality constraints.\end{Rem}

A straightforward consequence of Theorem~\ref{Teo:LimHLS} is that we can estimate $\ix{f^2\,\ninf}$ in terms of $\mathsf Q_1[f]=\ix{f\,(f-g\,\cinf)\,\ninf}$. This result has been announced in \cite{CD2012Cras}.
\begin{Corollary}\label{Cor:LimHLS} For the same value of $\Lambda>1$ as in Theorem~\ref{Teo:LimHLS}, we have
\[
\ix{f^2\,\ninf}\le\frac\Lambda{\Lambda-1}\,\mathsf Q_1[f]
\]
for any $f\in L^2(\R^2,\ninf\,dx)$ such that \eqref{Cdt:Average} holds.
\end{Corollary}
\begin{proof} We may indeed write
\[
\mathsf Q_1[f]=\frac{\Lambda-1}\Lambda\,\ix{f^2\,\ninf}+\frac 1\Lambda\,\(\ix{|f|^2\,\ninf}-\Lambda\ix{f\,\ninf\,(-\Delta)^{-1}\,(f\,\ninf)}\)
\]
and use the fact that the last term of the right hand side is nonnegative.\end{proof}

\section{The large time behavior}\label{Sec:PrfThmMain}

This section is devoted to the proof of Theorem~\ref{Thm:Main}. Our approach is guided by the analysis of the evolution equation corresponding to the linearization of the Keller-Segel system: see Section~\ref{Sec:linearized}. The key estimates for the nonlinear evolution problem have been stated in Theorem~\ref{Teo:Gap}, Theorem~\ref{Teo:LimHLS}, and Corollary~\ref{Cor:LimHLS}. Nonlinear terms are estimated using Corollary~\ref{Cor:UniformCV}.

\subsection{A linearized evolution problem}\label{Sec:linearized}

We recall that the restriction of $\L$ to $\mathcal D(\mathsf L_1)$ is a self-adjoint operator with domain $\mathcal D(\mathsf L_2)$, such that
\[
\langle f,\L\,f\rangle=-\,\mathsf Q_2[f]\quad\forall\;f\in\mathcal D(\mathsf L_2)\;.
\]
and $\mathrm{Ker}(\L)\cap\mathcal D(\mathsf L_2)=\{0\}$

By Proposition~\ref{Prop:LisSelfAdjoint}, any solution $(t,x)\mapsto f(t,x)$ of the linearized Keller-Segel model
\ben
\left\{\begin{array}{lcr}
\dr{f}{t} =\L\,f& x\in\R^2\,, &t>0\;,\\[6pt]
-\Delta (g\,\cinf) = f\,\ninf & x\in \R^2\,, &t>0\;,
\end{array}\right.
\een
has an exponential decay, since we know that
\[
\frac d{dt}\langle f(t,\cdot),f(t,\cdot)\rangle=2\,\langle f(t,\cdot),\,\L\,f(t,\cdot)\rangle\;,
\]
that is
\[\label{Eqn:linearizedDecay}
\frac d{dt}\,\mathsf Q_1[f(t,\cdot)]=-\,2\,\mathsf Q_2[f(t,\cdot)]\le-\,2\,\mathsf Q_1[f(t,\cdot)]
\]
by Theorem~\ref{Teo:Gap}. Hence we obtain
\[
\mathsf Q_1[f(t,\cdot)]\le\mathsf Q_1[f(0,\cdot)]\,e^{-\,2\,t}\quad\forall\;t\in\R^+\,.
\]
Here we adopt the usual convention that $\mathsf Q_2[f]=+\infty$ for any $f\in\mathcal D(\mathsf L_1)\setminus\mathcal D(\mathsf L_2)$.

\subsection{Proof of Theorem~\ref{Thm:Main}}\label{Sec:PrfMain}

As in \cite{Blanchet2010533}, Eq.~\eqref{RKS} can be rewritten in terms of $f=(n-\ninf)/\ninf$ and $g=(c-\cinf)/\cinf$~in the form of \eqref{pblin}, that is
\[
\dr{f}{t} =\L\,f-\frac 1\ninf\,\nabla\left[\ninf\,f\,\nabla(g\,\cinf)\right]\;.
\]
The computation for the linearized problem established in Section~\ref{Sec:linearized} can be adapted to the nonlinear case and gives
\[
\frac d{dt}\,\mathsf Q_1[f(t,\cdot)]=-\,2\,\mathsf Q_2[f(t,\cdot)]+\,2\ix{\nabla(f-g\,\cinf)\,f\,\ninf\cdot\nabla(g\,\cinf)}\;,
\]
with $-\,2\,\mathsf Q_2[f(t,\cdot)]\le-\,2\,\mathsf Q_1[f(t,\cdot)]$ according to Theorem~\ref{Teo:Gap}.

As noted in Section~\ref{Sec:FunctSetting}, Condition~\eqref{Cdt:Average} is satisfied, which means that $\langle f,f_{0,0}\rangle=0$. Indeed, using the explicit expression of $f_{0,0}=\partial_M\log n_{\infty,M}$, we recall that
\begin{eqnarray*}
\langle f(t,\cdot),f_{0,0}\rangle&=&\ix{f(t,\cdot)\,f_{0,0}\,\ninf}-\ix{f(t,\cdot)\,\ninf\,(-\Delta)^{-1}(f_{0,0}\,\ninf)}\\
&=&\ix{f(t,\cdot)\,(\partial_M\ninf-\ninf\,\partial_M\,c_\infty)}\\
&=&\frac1µ\(1+\ix{\partial_M\,c_\infty\,\ninf}\)\,\ix{f(t,\cdot)\,\ninf}=0
\end{eqnarray*}
because of the mass conservation.

To get an estimate on the asymptotic behaviour, we have to establish an estimate of the last term of the right hand side, which is cubic in terms of $f$. For this purpose, we apply H\"older's inequality to get
\[
\(\ix{\nabla(f-g\,\cinf)\,f\,\ninf\cdot\nabla(g\,\cinf)}\)^2\le \mathsf Q_2[f]\,\ix{f^2\,\ninf}\,\nrm{\nabla(g\,\cinf)}\infty^2\;.
\]
Using Corollary~\ref{Cor:LimHLS}, we find that the right hand side can be bounded by
\[
\frac\Lambda{\Lambda-1}\,\mathsf Q_1[f]\,\mathsf Q_2[f]\,\nrm{\nabla(g\,\cinf)}\infty^2
\]
and $\lim_{t\to\infty}\nrm{\nabla(g(t,\cdot)\,\cinf)}\infty=0$ according to Corollary~\ref{Cor:UniformCV}. Therefore, there exists a positive continuous function $t\mapsto\delta(t)$ with $\lim_{t\to\infty}\delta(t)=0$ such that
\[
\frac d{dt}\,\mathsf Q_1[f(t,\cdot)]\le-\,2\,\mathsf Q_2[f(t,\cdot)]+\delta(t)\,\sqrt{\mathsf Q_1[f(t,\cdot)]\,\mathsf Q_2[f(t,\cdot)]}\le(\delta(t)-\,2)\,\mathsf Q_2[f(t,\cdot)]\;,
\]
where the last inequality is a consequence of Theorem~\ref{Teo:Gap}. For some $t_*>0$, large enough, we have that $\delta(t)-\,2\le0$ for any $t>t_*$ and by Theorem~\ref{Teo:Gap} again, we get that
\[
\frac d{dt}\,\mathsf Q_1[f(t,\cdot)]\le-\,(2-\delta(t))\,\mathsf Q_1[f(t,\cdot)]\quad\forall\,t>t_*\;,
\]
thus proving that $\lim_{t\to\infty}\mathsf Q_1[f(t,\cdot)]=0$. Now we can give a refined estimate of the decay of $\nrm{\nabla(g\,\cinf)}\infty$. Applying Lemma~\ref{Lem:Linfty} with $\varepsilon=1$, we get that
\[
\nrm{\nabla(g\,\cinf)}\infty\le C\,\(\nrm{f\,\ninf}1+\nrm{f\,\ninf}3\)\,.
\]
Using H\"older's inequality, we find that
\begin{eqnarray*}
&&\nrm{f\,\ninf}1\le\sqrt M\,\nrm{f\,\sqrt\ninf}2\;,\\
&&\nrm{f\,\ninf}3\le\nrm{f\,\sqrt\ninf}2^{2/3}\,\nrm{f\,\ninf}\infty^{1/3}\,\nrm{\ninf}\infty^{1/3}\;.
\end{eqnarray*}
According to Corollary~\ref{Cor:LimHLS}, $\nrm{f\,\sqrt\ninf}2^2$ is bounded by $\frac\Lambda{\Lambda-1}\,\mathsf Q_1[f]$, so that
\[
\delta(t)\le O(\mathsf Q_1[f(t,\cdot)]^{4/3})\quad\mbox{as}\quad t\to\infty\;.
\]
As a consequence, we finally get that
\[
\limsup_{t\to\infty}e^{2\,t}\,\mathsf Q_1[f(t,\cdot)]<\infty\;,
\]
which completes the proof of Theorem~\ref{Thm:Main}.\qed

\begin{Rem} From the estimate of Theorem~\ref{Thm:Main}, we can deduce uniform rate of convergence. Indeed, by repeating the computations of Section~\ref{Sec:Uniform}, we may now give a refined estimate of $\mathcal R(t)$. From Theorem~\ref{Thm:Main}, we deduce the estimates
\[
\nrm{n(t,\cdot)-\ninf}1\le\(\nrm{\ninf}1\ix{\frac{|n(s,\cdot)-\ninf|^2}{\ninf}}\)^\frac 12\le\frac{\sqrt{C\,M}}{2\pi}\,\frac{e^{-t}}{1-e^{-2t}}\;,
\]
\[\nrm{n(s,\cdot)-\ninf}p\le\(\nrm{|n(s,\cdot)-\ninf|^{p-2}\,\ninf}\infty\ix{\frac{|n(s,\cdot)-\ninf|^2}{\ninf}}\)^\frac 1p=O(e^{-2s/p})\;.
\]
Moreover, the equivalence of $\mathsf Q_1[f]$ with $\ix{f^2\,\ninf}$ and Theorem~\ref{Thm:Main} shows that
\[
\nrm{\nabla c(s,\cdot)-\nabla \cinf}2^2=O(e^{-2s})\quad\mbox{as}\quad s\to\infty\;,
\]
and, as a consequence of Corollary~\ref{Cor:UniformCV},
\[
\nrm{\nabla c(s,\cdot)-\nabla \cinf}q=O(e^{-s/q})\quad\mbox{as}\quad s\to\infty\;.
\]
Hence with $p=2\,q=\frac{3\,r}{r-1}$ we have shown that 
\[
\nrm{n(t,\cdot)-\ninf}\infty=O\(e^{-\frac{2\,(r-1)}{3\,r}t}\)\quad\mbox{as}\quad t\to\infty\;.
\]
Reinjecting this estimate in the above ones, we can even get an improved rate of convergence.
\end{Rem}

\section{Concluding remarks}\label{Sec:Conclusion}

As a conclusion, let us summarize the main novelties of the method developed in this paper.
\begin{enumerate}
\item Symmetrization and comparison with the symmetrized problem are applied in rescaled variables, and not in the original ones as was done in previous papers. As a consequence, algebraic decay estimates in the original variables are obtained, and not only bounds.
\item These estimates are sufficient to reduce the nonlinear problem to a linearized one. Still one has to understand the spectrum of the linearized operator. 
\item As a consequence of the invariances of the original equation, the three lowest modes can be identified. They are independent of the mass (which is not the case for the other ones). 
\item Our results strongly rely on new, sharp functional inequalities. Getting sharp logarithmic Hardy-Littlewood-Sobolev inequalities in the subcritical range of masses is a consequence of previous papers, but the fact that it has to be done in self-similar variables (after rescaling) has to be emphasized. The Onofri counterpart, which is obtained by a Legendre transform, has been established recently in \cite{CD2012Cras} and is of independent interest. Linearization of both inequalities allows us to establish sharp spectral gap inequalities (and improvements under orthogonality constraints) that are crucial for this paper and may prove to be useful for other mean field type problems.
\item As far as we know, proving that the linearized operator is self-adjoint in the functional space corresponding to the linearization of the free energy (which is a Lyapunov functional) is a natural idea, but has not been exploited yet in this class of problems. The scalar product is new, non-trivial, and we are not aware of the use of similar tools in problems with a mean field term.
\end{enumerate}

\bigskip\begin{spacing}{0.8}\noindent{\it Acknowledgments.\/} {\small The authors acknowledge support by the ANR projects \emph{CBDif-Fr} and \emph{EVOL} (JD), and by the MathAmSud project \emph{NAPDE} (JC and JD). They also thank the anonymous referees who have done a tremendous work in checking all details of the paper and led to a significant improvement of the final version.}
\par\medskip\noindent{\scriptsize\copyright\,2013 by the authors. This paper may be reproduced, in its entirety, for non-commercial purposes.}\end{spacing}
\bibliographystyle{siam}\small

\end{document}